\documentclass[11pt,reqno]{amsart}
\usepackage[skip=8pt plus1pt, indent=20pt]{parskip}
\usepackage{amsmath,amssymb,amsfonts,amsthm}
\usepackage{mathrsfs}
\usepackage{hyperref}
\usepackage{cite}
\usepackage[sc]{mathpazo}
\usepackage{graphics}
\usepackage{graphicx}
\usepackage{empheq}
\usepackage{cases}
\usepackage{enumitem}
\usepackage{comment}
\usepackage{appendix}
\usepackage{marginnote}
\usepackage[left=30mm,right=30mm,top=30mm,bottom=30mm,heightrounded, marginparwidth=2cm, marginparsep=0.1cm]{geometry}
\usepackage{color}
\usepackage{enumitem}
\usepackage[T1]{fontenc}
\usepackage{subfig}
\usepackage{outlines}
\usepackage{thmtools}
\usepackage{commath}
\usepackage[font=footnotesize]{caption}
\usepackage{mathtools}

\usepackage[normalem]{ulem}

\newcommand{\Z}{\mathbb{Z}}
\newcommand{\N}{\mathbb{N}}
\newcommand{\R}{\mathbb{R}}
\newcommand{\C}{\mathbb{C}}
\renewcommand{\L}{\mathsf{L}^2}

\renewcommand{\H}{\mathscr{H}}

\newcommand{\e}{\mathrm{e}}
\newcommand{\al}{\alpha}
\newcommand{\be}{\beta}
\newcommand{\ga}{\gamma}
\newcommand{\gaga}{_{\ga,\wt\ga}}
\renewcommand{\d}{\,\mathrm{d}}

\newcommand{\cupl}{\bigcup\limits}

\newcommand{\wt}{\widetilde}
\newcommand{\wh}{\widehat}
\newcommand{\ceq}{\coloneqq}

\theoremstyle{plain}
\newtheorem{theorem}{Theorem}[section]
\newtheorem*{theorem*}{Theorem}
\newtheorem{lemma}[theorem]{Lemma}
\newtheorem*{lemma*}{Lemma}
\newtheorem{proposition}[theorem]{Proposition}
\newtheorem{corollary}[theorem]{Corollary}
\theoremstyle{remark}
\newtheorem{remark}[theorem]{Remark}
\newtheorem*{remark*}{Remark}

\newtheorem*{example*}{Example}
\theoremstyle{definition}

\numberwithin{equation}{section}
\numberwithin{figure}{section}
\newcommand{\black}{\color{black}}

\title{Quantum lattice transport along an infinitely extended perturbation}

\author[M.~Baradaran]{Marzieh Baradaran}
\address[M.~Baradaran]{Department of Physics, Faculty of Science, University of Hradec Kr\'{a}lov\'{e}, Rokitansk\'eho 62, 50003 Hradec Kr\'alov\'e, Czech Republic}
\email{marzieh.baradaran@uhk.cz}

\author[P.~Exner]{Pavel Exner}
\address[P.~Exner]{Doppler Institute for Mathematical Physics and Applied Mathematics, Czech Technical University,  B\v rehov\'a 7, 11519 Prague, Czechia \\ and Department of Theoretical Physics, NPI, Academy of Sciences, Hlavn\'{\i}130, 25068 \v{R}e\v{z} near Prague, Czechia}
\email{exner@ujf.cas.cz}

\author[A.~Khrabustovskyi]{Andrii Khrabustovskyi}
\address[A.~Khrabustovskyi]{Department of Physics, Faculty of Science, University of Hradec Kr\'{a}lov\'{e}, Rokitansk\'eho 62, 50003 Hradec Kr\'alov\'e, Czech Republic}
\email{andrii.khrabustovskyi@uhk.cz}

\begin{document}
	
	\clearpage
	\maketitle
	\thispagestyle{empty}
	
	\captionsetup[figure]{labelfont={bf},labelformat={default},labelsep=period,name={Fig.}}

	\begin{abstract}
		We consider a periodic quantum graph in the form of a rectangular lattice with the $\delta$-coupling of strength $\gamma$ in the vertices perturbed by changing the latter at an infinite straight array of vertices to a $\widetilde\gamma\ne\gamma$. We analyze the band spectrum of the system and show that it remains preserved as a set provided $\widetilde\gamma>\gamma>0$ while for all the other combinations additional band appear in some or all gaps of the unperturbed system. We also prove that for a randomly chosen positive energy, the probability of existence of a state exponentially localized in the vicinity of the perturbation equals $\frac12$.
	\end{abstract}
		

	\section{Introduction}
	\label{section:intro}
	
    Control of particle transport is an essential feature of numerous quantum systems. It concerns, in particular, \emph{waveguide effects} when the motion occurs in a given direction while in the perpendicular plane the system state is localized. This can be achieved in various ways, either by hard confinement or using a suitably chosen potential \cite{EK15}. The latter can be of different sorts, for instance a regular or singular potential `channel' or a more complicated `guide' with no classical analogue such as the infinite \emph{straight polymer} models \cite[Secs.~III.1.5 and III.4]{AGHH05} with arrays of point potentials. Sometimes there is even no transport without the appropriate `guiding' interaction as it is the case with the well-known \emph{Iwatsuka model} \cite{Iw85} where the motion is due to a translationally invariant variation of the magnetic field.

    Our goal here is to investigate a model of a waveguide effect in quantum graphs, specifically a perturbed rectangular lattice with the $\delta$-coupling at the vertices, with the perturbation consisting of modifying the coupling constant along a infinite long straight chain of vertices. Properties of the unperturbed periodic lattice are well known; it has a band spectrum with absolutely continuous bands, plus some flat ones in case that the lattice sides are commensurate. Recently a Bethe-Sommerfeld-type behavior, i.e. the finite number of open gaps, was observed in this model under an appropriate Diophantine condition \cite{ET17}.
	
	Using the translational invariance of the perturbation one can analyze its effect using Floquet-Bloch technique. We prove that in the vicinity of unperturbed band edges, there are families of states which are transversally
    exponentially localized; their energies may or may not be embedded in the unperturbed spectrum. We prove that if the basic $\delta$-interaction is repulsive and the perturbation strengthens the repulsion, the spectrum as a set does not change, while in all the other cases additional bands appear in (some or all) gaps of the unperturbed spectrum. What is even more important, we show that the family of the perturbation-induced states is significant; the probability that such a state exists at a randomly chosen positive energy is $\frac12$.
	
	Let us briefly describe the contents of the paper. In the next section we recall the properties of the model in the unperturbed case. Section~\ref{perturbed,section} is devoted to a detailed discussion of the spectrum which
    arises due to the perturbation; the main results are collected in Theorem~\ref{thmGen} at the end of this section. Finally, in Section~\ref{high,en,section} we inspect the high-energy behavior and compute the measure of the perturbation-induced spectral component.

	
	\section{Spectrum of the unperturbed periodic system}
	\label{section:2}
	
	We begin the discussion by reviewing spectral properties of the Hamiltonian $\mathscr{H}_\ga$ of a particle confined to a rectangular lattice with a $\delta$-coupling. They have been already investigated in detail \cite{Ex96,EG96,ET17}, thus we only recall the results we need here.
	
	\begin{figure}[h]
		\centering
		\subfloat[ The  periodic (unperturbed) graph  with a $\delta$-coupling of the strength $\gamma$ at all vertices]{{\includegraphics[scale=0.55]{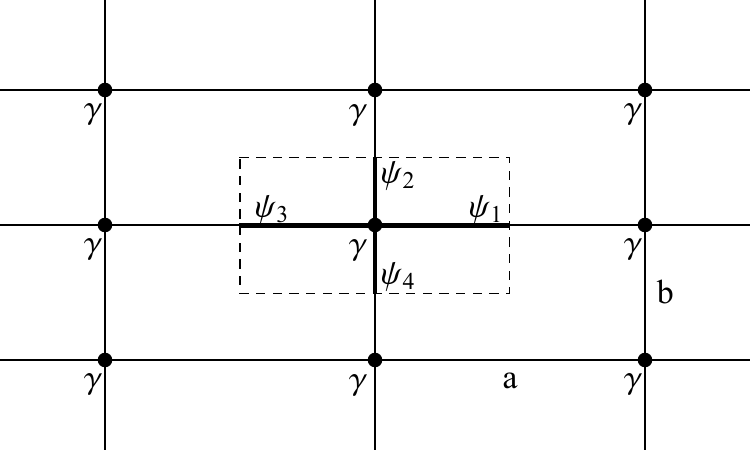} }
			\label{UnitCell-2}}
		\qquad
		\subfloat[ The perturbed graph in which, in a straight chain of vertices, the $\delta$-coupling strength $\gamma$ is replaced by $\widetilde{\gamma}$ ]{\includegraphics[scale=0.55]{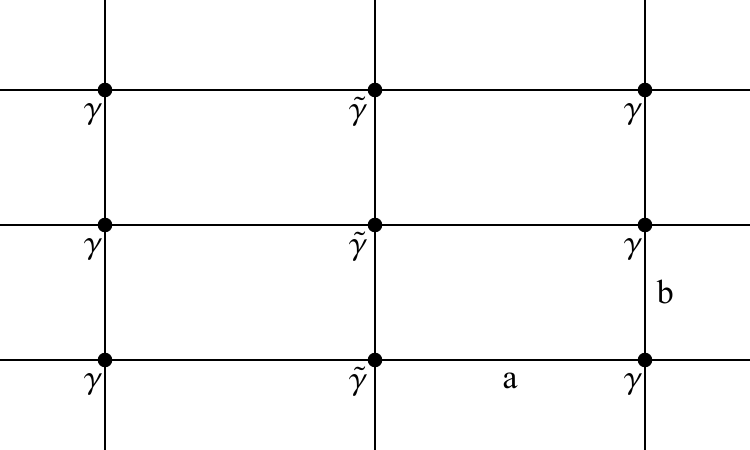} \label{UnitCell-3}}
		\caption{A rectangular lattice }
	\end{figure}
	
	An elementary cell of the rectangular lattice $\Gamma$ of sides $a,b>0$
	consists of a single vertex and four outgoing edges (Fig.~\ref{UnitCell-2}).
	To describe the spectrum $\sigma(\H_\ga)$ of $\H_\ga\,$, one utilizes the Floquet–Bloch theory (see, e.g., \cite[Chapt.~4]{BK13})  reducing the analysis of $\sigma(\H_\ga)$ to the spectral analysis of
	the fiber operators $\H_\ga^{\theta_1,\theta_2}$ on an elementary cell, which are labeled by the quasimomentum components $\theta_1,\theta_2\in [-\pi,\pi]$ indicating how is the phase of the wavefunctions related at the opposite ends of the elementary cell (cf.~\eqref{PerFloq} below).

	We start from the investigation of \emph{the positive spectrum}.
	Since the Hamiltonian acts on the graph edges as the negative Laplacian, the eigenfunctions of $\H_\ga^{\theta_1,\theta_2}$ at energy $E>0$ are linear combinations of the functions $\e^{\pm ikx}$, where $k=E^{1/2}$.
	Choosing the coordinates on the graph edges to increase from left to right and from bottom to top, we employ the following \emph{Ansätze} for the wave function components $\psi_k$  indicated in Fig.~\ref{UnitCell-2}:
	\begin{align}\label{PerAnsatz}
		& \psi_{1}(x)=\alpha_{1}^{+}\e^{ikx}+\alpha_{1}^{-}\e^{-ikx},\qquad x\in[0,\tfrac a2],  \nonumber \\[3pt]
		& \psi_{2}(x)=\alpha_{2}^{+}\e^{ikx}+\alpha_{2}^{-}\e^{-ikx},\qquad x\in[0,\tfrac b2],\\[3pt]
		& \psi_{3}(x)=\alpha_{3}^{+}\e^{ikx}+\alpha_{3}^{-}\e^{-ikx},\qquad x\in[-\tfrac a2,0], \nonumber\\[3pt]
		& \psi_{4}(x)=\alpha_{4}^{+}\e^{ikx}+\alpha_{4}^{-}\e^{-ikx},\qquad x\in[-\tfrac b2,0].\nonumber
	\end{align}
	The Floquet-Bloch decomposition requires to impose the following conditions at the free ends of the elementary cell,
	\begin{align}\label{PerFloq}
		\begin{array}{ll}
			\psi_1 (\tfrac a2 ) =\e^{i \theta_{1} } \psi_3 (-\tfrac a2 ),\; &
			\psi_1'(\tfrac a2 ) =\e^{i \theta_{1}} \psi_3' (-\tfrac a2),
			\\[6pt]
			\psi_2 (\tfrac b2 ) =\e^{i \theta_{2} } \psi_4 (-\tfrac b2 ),\; &
			\psi_2' (\tfrac b2) =\e^{i \theta_{2}} \psi_4' (-\tfrac b2),
		\end{array}
	\end{align}
	with $\theta_1$ and $\theta_2$ running through $[-\pi,\pi]$, the Brillouin zone; substituting \eqref{PerAnsatz} into \eqref{PerFloq} allows us to express $\alpha_{1}^{\pm}$ and $\alpha_{2}^{\pm}$ in terms of $\alpha_{3}^{\pm}$ and $\alpha_{4}^{\pm}$, respectively:
	\begin{align}\label{coeffFloqPer}
		\alpha_{1}^{+}  =\alpha_{3}^{+} \,\cdot\, \e^{i(\theta_1- a k)}, \quad
		\alpha_{1}^{-}  =\alpha_{3}^{-} \,\cdot\,\e^{i(\theta_1+ a k)},\quad
		\alpha_{2}^{+} =\alpha_{4}^{+}  \,\cdot\, \e^{i(\theta_2- b k)},\quad
		\alpha_{2}^{-}  =\alpha_{4}^{-}  \,\cdot\, \e^{i(\theta_2+ b k)}.
	\end{align}
	Furthermore, the $\delta$-coupling of the strength $\gamma$
	implies the following conditions at the vertex:
	\begin{align}
		\label{perEqs}
		\psi_1(0)=\psi_2(0)=\psi_3(0)=\psi_4(0), \qquad
		& \psi_1'(0)+\psi_2'(0)-\psi_3'(0)-\psi_4'(0)=\gamma\,\psi_1(0).
	\end{align}
	Combining \eqref{PerAnsatz}, \eqref{coeffFloqPer} and \eqref{perEqs},   we get a system of four linear equations for the coefficients $\alpha_{3}^{\pm}$ and $\alpha_{4}^{\pm}$; computing the corresponding determinant and neglecting the non-vanishing multiplicative factor $\e^{i \left(\theta _1+\theta _2\right)} k^2$, we arrive at the spectral condition
	\begin{equation}\label{spectral:condition}
		\left( \tau_1 - \cos ak   \right)\sin bk   +
		\left(\tau_2-\cos b k\right)\sin a k =
		{\gamma\over 2k} \,\sin ak\,\sin bk,
	\end{equation}
	where $\tau_k\ceq\cos\theta_k,\ k=1,2$.
	Thus,   $E>0$ belongs to $\sigma(\H_\ga)$ iff
	$k=E^{1/2}$ satisfies relation \eqref{spectral:condition} with some $\tau_1,\tau_2\in [-1,1]$.
	
	We distinguish a particular set on momentum values denoting
	\begin{gather*}
		\Xi_a\ceq\{k>0:\ \sin ka=0\},
		\quad
		\Xi_b\ceq\{k>0:\ \sin kb=0\},
		\quad
		 \Xi\ceq\Xi_a\cup\Xi_b\,;
	\end{gather*}
 If $k$ belongs to $\Xi $, the condition \eqref{spectral:condition} can be satisfied for suitable chosen $\tau_1, \tau_2$. This does not guarantee, however, the existence of $L^2$ eigenfunctions of the operator $\mathscr{H}_\ga$. To this end one needs the `loop bound states' \cite{Ku05}; recall that in contrast to usual Schr\"odinger operators quantum graphs do not have the unique continuation property. Such bound states exist at the \emph{intersection} of the two sets, in other words	
	\begin{gather}
		\label{specXi}
		\{E>0:\ E^{1/2}\in \Xi_a\cap\Xi_b\}\subset\sigma_{\rm point}(\H_\ga).
	\end{gather}
	Specifically, the numbers $k^2=\frac{n^2\pi^2}{a^2}=\frac{m^2\pi^2}{b^2}$ are eigenvalues of infinite multiplicity, flat bands in physicist's terminology, provided that $\frac ba=\frac mn$ holds for some $m,n\in\mathbb{N}$. Next, assuming that $k\notin\Xi$,
	we can rewrite condition \eqref{spectral:condition} as follows,
	\begin{equation}\label{spectral:condition:symmetric}
		{\gamma\over 2k}=F(\tau_1,\tau_2,k),\quad\text{where }
		F(\tau_1,\tau_2,k)\ceq {\tau_1 - \cos ak  \over \sin ak} +
		{\tau_2 - \cos bk  \over \sin bk}.
	\end{equation}
	For any fixed $k\in (0,\infty)\setminus\Xi$,
	the function $F(\tau_1,\tau_2,k)$ depends continuously on $\tau_k\in [-1,1]$.
	Furthermore, one can   show (see \cite[Sec.~3.2]{Ex96} for more details)
	that
	\begin{gather*}
		\max_{\tau_k\in [-1,1]}F(\tau_1,\tau_2,k)=
		F_+(k),\quad \min_{\tau_k\in [-1,1]}F(\tau_1,\tau_2,k)=
		F_-(k),
	\end{gather*}
	where
	\begin{gather*}
		F_+(k)\ceq \tan\left(\frac {ka}2-\frac{\pi}{2}\left\lfloor\frac {ka}\pi\right\rfloor\right)+ \tan\left(\frac {kb}2-\frac{\pi}{2}\left\lfloor\frac {kb}\pi\right\rfloor\right),
		\\
		F_-(k)\ceq -\cot\left(\frac {ka}2-\frac{\pi}{2}\left\lfloor\frac {ka}2\right\rfloor\right)-\cot\left(\frac {kb}2-\frac{\pi}{2}\left\lfloor\frac {kb}\pi\right\rfloor\right),
	\end{gather*}
	in which $\lfloor x \rfloor $ stands for the greatest integer less than or equal to $x$. It is easy to see that 	
	\begin{gather}
		\pm F_\pm(k)\ge 0,\quad F_\pm\in C^\infty(\R\setminus\Xi),\quad
		\lim_{k\to k_0\mp 0}F_\pm(k)=\infty,\ \forall k_0\in\Xi, \nonumber \\[-1.5em]  \\[-.5em]
		F_\pm\text{ monotonically increase on each  interval }(k_-,k_+)\subset (0,\infty)\setminus\Xi
		\text{ with }k_\pm\in\Xi. \nonumber
	\end{gather}
We denote
	$$\Sigma_\pm \ceq
	\left\{k\in (0,\infty)\setminus\Xi:\ \pm{\ga\over 2k}\leq \pm F_\pm(k)\right\},\quad
	\gamma_*=-4(a^{-1}+b^{-1}).
	$$
	
	The above properties of $F_\pm$ together with \eqref{specXi} and \eqref{spectral:condition:symmetric}
	imply the following statement:

	\begin{proposition}\label{prop:spec:positive}
		
		The positive part of the spectrum of $\H_\ga$ exhibits the following properties:
		
		\begin{enumerate}
			\setlength{\itemsep}{4pt}
			
			\item[(i)]
			The number $E>0$  belongs to the spectrum of $\H_\ga$ with $\pm\ga\ge 0$  iff
			${E}^{1/2}\in \Sigma_\pm  \cup\Xi.$
			
			\item[(ii)]
			In the rectangle sides are commensurate, flat bands occur at the energies $k^2=\frac{n^2\pi^2}{a^2}=\frac{m^2\pi^2}{b^2}$ iff $\frac ba=\frac mn$ holds for some $m,n\in\mathbb{N}$.

			\item[(iii)]
			The spectrum is absolutely continuous (\emph{ac}) in $\Sigma_\pm$ and possible flat bands lie at the edges of \emph{ac} bands.
			
			\item[(iv)]
			The gaps, $(0,\infty)\setminus (\Sigma_\pm  \cup\Xi)$, are a union of
			countably many open intervals  whose closures are pairwise disjoint.
			If the number of these intervals is infinite, they accumulate at infinity.
	
			\item[(v)]
			For $\ga\in (\ga_*,0]$  one has $\inf(\sigma(\H_\ga)\cap (0,\infty))=0$. On the other hand, if $\gamma>0$ or $\ga\le\gamma_*$, one has $\inf(\sigma(\H_\ga)\cap (0,\infty))>0$.
		
			\item[(vi)]  If $(E,E+\delta)\subset\R\setminus \sigma(\H_\ga)$
			with $E \in \sigma(\H_\ga)\cap (0,\infty)$ and $\delta>0$, then
			\begin{gather}\label{gap:leftedge}
				{\ga\over 2E^{1/2}}=F_-(E^{1/2})\;\text{ if }\;\ga<0,
				\quad
				E^{1/2} \in \Xi\;\text{ if }\;\ga>0,
			\end{gather}
			
			\item[(vii)] If $(E-\delta,E)\subset\R\setminus \sigma(\H_\ga)$
			with $E \in \sigma(\H_\ga)\cap (0,\infty)$ and $\delta>0$, then
			\begin{gather}\label{gap:rightedge}
				E^{1/2} \in \Xi\;\text{ if }\;\ga<0,
				\quad
				{\ga\over 2E^{1/2}}=F_+(E^{1/2})\;\text{ if }\;\ga>0.
			\end{gather}
			
		\end{enumerate}

	\end{proposition}
		
	\begin{remark} \label{rem:Bethe-Sommerfeld}	
    The number of open gaps depends on the ratio $\frac ba$. It is generically infinite; this happens both in the commensurate case as well if the ratio is a Liouville number. On the other hand, if it is of a Diophantine type, badly approximable by rationals in the sense that the sequence of its continued-fraction expansion coefficients is bounded, the spectrum has for appropriate values of $\ga$ a finite number of open gaps or no gaps at all \cite{Ex96,ET17}, cf. Figs.~\ref{threefigsRectGolden} and \ref{threefigsRectGolden3}.
    \end{remark}

	The non-positive part of the spectrum can be determined by replacing the momentum variable $k$ by $i\kappa$ with $\kappa>0$.
	Omitting the details, we present only the final conclusion. For $\gamma<0$, the equation
	$$
	{\ga\over 2\kappa}=-\tanh\left(\frac {\kappa a}2\right)-\tanh\left(\frac {\kappa b}2\right),
	$$
	has a unique negative root  $\kappa_1$,
	and for $\gamma<\gamma_*$ the equation
	$$
	{\ga\over 2\kappa}=-\coth\left(\frac {\kappa a}2\right)-\coth\left(\frac {\kappa b}2\right)
	$$
	possesses a unique negative root $\kappa_2$.

	\begin{proposition}\label{prop:spec:negative}
		For $\gamma\ge 0$ the non-positive part of the spectrum of $\H_\ga$ is empty.
		For $\gamma<0$ one has
		$$
		\sigma(\H_\ga)\cap (-\infty,0]=[E_1,E_2],
		$$
		where
		$$
		E_1=-\kappa_1^2\,,\quad
		E_2=
		\begin{cases}
			-\kappa_2^2\,,&\ga<\ga_*,\\
			0,&\ga_*\le \ga<0.
		\end{cases}
		$$
	\end{proposition}
	
	From Propositions~\ref{prop:spec:positive}--\ref{prop:spec:negative}
	we get the following corollary.
	
	\begin{corollary}
		Zero belongs to the spectrum of $\H_\ga$ iff $\gamma\in [\gamma_*,0]$.
		Furthermore, if $\ga\in (\gamma_*,0)$, it is an internal point of $\sigma(\H_\ga)$, while if $\gamma=\gamma_*$ there exists a $\delta>0$ such that
		$$
		(-\delta,0)\subset \sigma(\H_\ga),\quad
		(0,\delta)\subset \R\setminus\sigma(\H_\ga).
		$$
		
	\end{corollary}
	
	\begin{remark} \label{rem:periodic}
		For further consideration it is convenient to rewrite \eqref{spectral:condition:symmetric}
		as follows:
		\begin{equation}\label{sc,per,gen}
			\tau_1+\frac{\sin ka }{\sin kb}\,\tau_2=\frac{\gamma \, \sin ka}{2 k}+\frac{ \sin k(a+b)}{\sin k b}.
		\end{equation}
		Then, using \eqref{sc,per,gen} and taking into account that $\lvert \tau_k \rvert \leq1$, $k=1,2$,
		one can easily show that the number $E=k^2>0$ belongs to  $\sigma(\H_{\gamma})$ if and only if $k$ satisfies
		\begin{equation}\label{per,band,con}
			\left\lvert  \frac{\gamma \, \sin ka}{2 k}+\frac{ \sin k(a+b)}{\sin k b} \right\rvert \leq  1+\left\lvert \frac{\sin ka }{\sin kb}\right\rvert .
		\end{equation}
		Similarly, a number $E=-\kappa^2<0$ belongs to $\sigma(\H_{\gamma})$ if and only if $\kappa$ satisfies
		\begin{equation}\label{per,neg,band,con}
			\left\lvert  \frac{\gamma \, \sinh \kappa a}{2 \kappa}+  \frac{\sinh \kappa (a+b)}{\sinh \kappa b} \right\rvert \leq  1+  \frac{\sinh \kappa a }{\sinh \kappa b}  .
		\end{equation}
		Furthermore, the properties \textit{(vi)}--\textit{(vii)} of Proposition~\ref{prop:spec:positive}
		can be reformulated as follows:
		if $(E-\delta,E)\subset\R\setminus \sigma(\H_\ga)$ (respectively, $(E,E+\delta)\subset\R\setminus \sigma(\H_\ga)$) with $E \in \sigma(\H_\ga)\cap (0,\infty)$ and $\delta>0$,
		then $k=E^{1/2}$ satisfies
		\begin{gather*}
                \begin{cases}
				 	k \in \Xi&\text{if }\;\ga<0
				\\
				k \notin \Xi\text{\quad and\quad }\left\lvert  \frac{\gamma \, \sin ka}{2 k}+\frac{ \sin k(a+b)}{\sin k b} \right\rvert =  1+\left\lvert \frac{\sin ka }{\sin kb}\right\rvert
				&\text{if }\;\ga>0
                \end{cases}
				\\
				 \left(\text{respectively, }
				\begin{cases}
				 	k \in \Xi&\text{if }\;\ga>0
				\\
				k \notin \Xi\text{\quad and\quad }\left\lvert  \frac{\gamma \, \sin ka}{2 k}+\frac{ \sin k(a+b)}{\sin k b} \right\rvert =  1+\left\lvert \frac{\sin ka }{\sin kb}\right\rvert
				&\text{if }\;\ga<0
                \end{cases}\ \right).
		\end{gather*}
In other words, if $\ga>0$ (respectively, $\ga<0$), each upper (respectively, lower) band edge is a square of some $k=\frac{n\pi}{a}$ or $\frac{n\pi}{b}$ with $n\in\mathbb{N}$ while the lower (respectively, upper) band edge is characterized by the condition $	\left\lvert  \frac{\gamma \, \sin ka}{2 k}+\frac{ \sin k(a+b)}{\sin k b} \right\rvert =  1+\left\lvert \frac{\sin ka }{\sin kb}\right\rvert$; for more details see \cite[claim (a2) in Sec.~III.A]{EG96}.
	\end{remark}

	The band-gap pattern of the spectrum of $\H_\ga$ is illustrated in Fig.~\ref{periodicPlots}.
	
		\begin{figure}[h]
		\centering
		\subfloat[\centering $b=4$ and $\gamma=-6$\label{perPlota}]{{\includegraphics[scale=0.45]{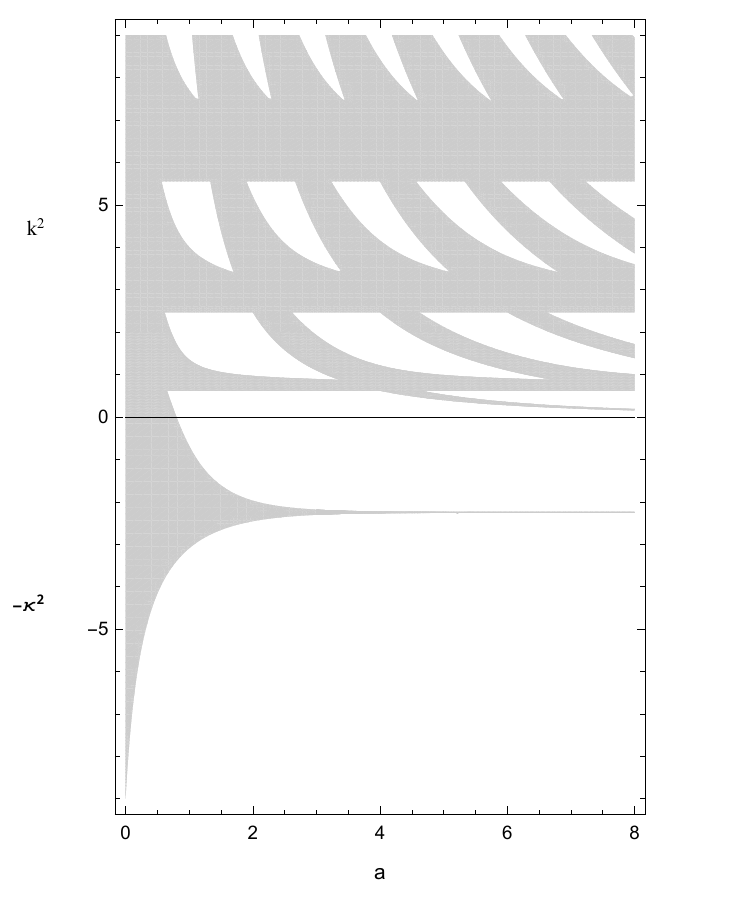} }
			\label{per1a}}
		\qquad
		\subfloat[\centering $a=3$ and $b=2$\label{perPlotb}]{\includegraphics[scale=0.45]{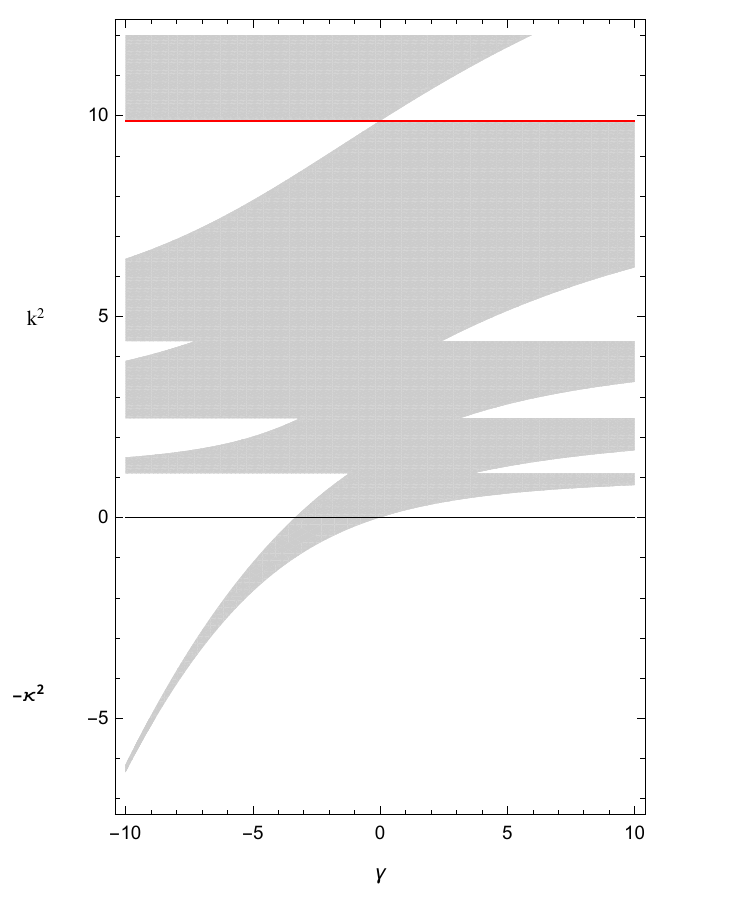} \label{pergamma}}
		\caption{Spectrum of $\H_\ga$ as a function of the edge length $a$ and the coupling constant $\gamma$. The gray shaded area represents the spectral bands. The red straight line corresponds to the flat band $(\frac{n\pi}{a})^2=(\frac{m\pi}{b})^2$ with $n=3$ and $m=2$.} \label{periodicPlots}
	\end{figure}


	\section{The perturbed system}
	\label{perturbed,section}
	
	We perturb the Hamiltonian $\H_\ga$ on an infinite straight chain of vertices as in Fig.~\ref{UnitCell-3}. In these vertices, we replace the $\delta$-coupling strength $\gamma$ by $\wt\gamma$.
	We denote the resulting Hamiltonian by $\H\gaga$.
	This operator is no longer $\Z^2$-periodic, however it remains to be periodic in the direction parallel to the perturbed chain of vertices, and we can utilize the Floquet-Bloch decomposition along this direction.
	The corresponding elementary cell (we denote it $\Gamma_0$)
	has the form of a periodic infinite double sided comb -- see Figure~\ref{UnitCell-transfer}.
	The analysis of $\sigma(\H\gaga)$ is now reduced to the spectral analysis
	of the operators $\H\gaga^{\theta_2}$ with $\theta_2\in [-\pi,\pi]$ acting in $L^2(\Gamma_0)$. This operator acts
	again as $-{\d^2\over \d x^2}$ on the edges being subject to the $\delta$-con\-ditions
	at the internal vertices and $\theta_2$-quasiperiodic conditions at the boundary vertices.
	The interaction strength at one internal vertex equals $\wt\gamma$,
	at the remaining internal vertices it is equal to $\gamma$.
	The Floquet-Bloch theory yields
	\begin{equation}\label{spectgaga}
			\sigma(\H\gaga)=\overline{\bigcup_{\theta_2\in [-\pi,\pi]}\,\sigma(\H\gaga^{\theta_2})}.
	\end{equation}
	
	\begin{figure}[h]
		\centering
		\includegraphics[scale=0.45]{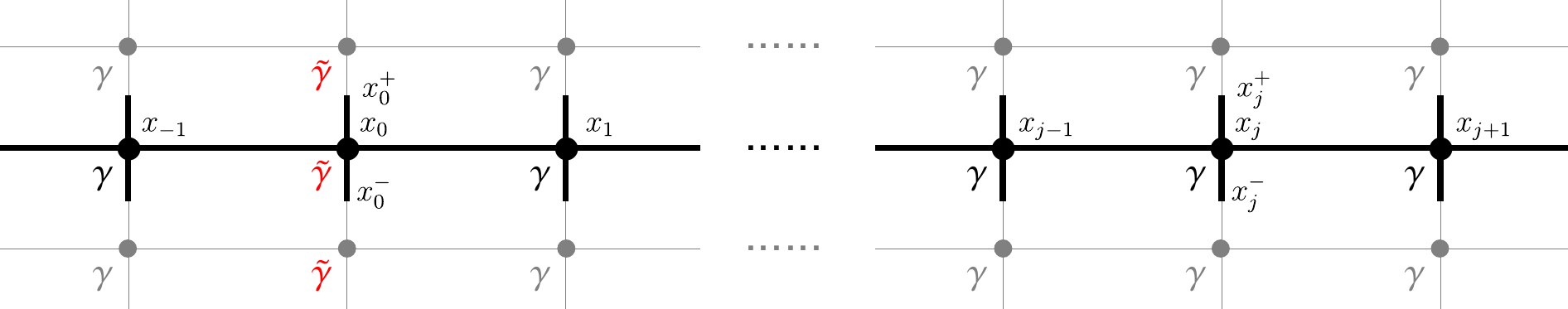}
		\caption{The infinite graph $\Gamma_0$ is specified by thick black lines.}
		\label{UnitCell-transfer}
	\end{figure}
	
	Along with $\H\gaga^{\theta_2}$ we also introduce the operator
	$\H_\ga^{\theta_2}$ putting
	$\H_\ga^{\theta_2}\ceq\H_{\ga,\ga}^{\theta_2}$
	so that \emph{all} the internal vertices have the interaction strength $\gamma$.
	Similarly to $\H_\ga $, the spectrum of $\H_\ga^{\theta_2}$ is purely essential and has band-gap structure. One has
	\begin{equation}\label{pertSpecUni}
	\sigma(\H_\ga)=
		\overline{\bigcup_{\theta_2\in [-\pi,\pi]}\,\sigma(\H_\ga^{\theta_2})}.
	\end{equation}
	In contrast to  $\H\gaga^{\theta_2}$ with $\ga\not=\wt\ga$,
	the operator $\H_\ga^{\theta_2}$ is periodic along the comb, and one can express its
	spectrum with help of the Floquet-Bloch decomposition, namely
	\begin{gather}
	\sigma(\H_\ga^{\theta_2})=\cupl_{\theta_1\in [-\pi,\pi]}\sigma(\H_\ga^{\theta_1,\theta_2}).
	\end{gather}
	Using the observation made in Remark~\ref{rem:periodic} we arrive at the following claim:

	\begin{proposition}\label{prop:per:comb}
		The number $E>0$ (respectively, $E<0$) belongs to the spectrum of $\H_\ga^{\theta_2}$ iff
		\begin{gather}
			|f_\ga^\tau(k)|\leq 1,\text{ where }k=E^{1/2},\ \tau=\cos\theta_2\\
			(\text{respectively, }|{ \wh{f}_\ga^\tau(\kappa)}|\leq 1,\text{ where }\kappa=|E|^{1/2},\ \tau=\cos\theta_2),
		\end{gather}	
		where 		
		the functions $f_\ga^\tau:(0,\infty)\setminus\Xi\to\R$ and $ { \wh f_\gamma^{\tau}:(0,\infty)\to\R}$ are given by
		\begin{gather}
			\label{sc,alter,adiffb}
			f_\gamma^{\tau}(k):=\frac{\gamma \, \sin ka}{2 k}+ \frac{ \sin k(a+b)}{\sin k b}  -\tau   \,\frac{\sin ka }{\sin kb}\,\\
			{ \wh f_\gamma^{\tau}(\kappa)}:=\frac{\gamma \, \sinh \kappa a}{2\kappa}+ \frac{ \sinh \kappa(a+b)}{\sinh \kappa b}  -\tau   \,\frac{\sinh \kappa a }{\sin \kappa b}\, .
		\end{gather}			
	\end{proposition}

    The coupling constant change in $\H\gaga^{\theta_2}$ amounts to a rank-one perturbation of the resolvent, hence the essential spectra of $\H_\ga^{\theta_2}$ and $\H\gaga^{\theta_2}$ coincide. Moreover, in each gap of the $\H_\ga^{\theta_2}$ spectrum, including the one below the spectral threshold, there is \emph{at most one eigenvalue} of $\H\gaga^{\theta_2}\,$ \cite[Sect.~8.3, Cor.~1]{We80}. Our goal is to investigate these eigenvalues, and to clarify, in particular, whether they belong to the gaps of the initial $\Z^2$-periodic operator $\H_\ga$.

	To proceed, we label the vertices of $\Gamma_0$ as follows:
	by $x_j$, $j\in\Z$, we indicate the internal vertices in such a way that
	$x_j$  and $x_{j+1}$ are linked by an edge, and the vertex $x_0$
	supports the interaction of the modified strength $\wt\ga$. By $x_j^\pm$, $j\in\Z$,
	we denote the vertices at the elementary cell boundary, see Figure~\ref{UnitCell-transfer}.
	
	The theorem below characterizes the possible eigenvalues lying in the gaps
	of $\sigma(\H_\ga^{\theta_2})$. Recall that by Proposition~\ref{prop:per:comb} $E\in (0,\infty)\setminus \sigma(\H_\ga^{\theta_2})$
	means that $|f_\gamma^{\tau}(E^{1/2})|>1$.
	
	\begin{theorem}\label{thm:per:comb}
		Let $E\in (0,\infty)\setminus \sigma(\H_\ga^{\theta_2})$.
		We set $k=E^{1/2}$, $\tau=\cos\theta_2$.
		Then $E\in\sigma_{\rm disc}(\H\gaga^{\theta_2})$ iff
		\begin{equation}\label{sol,gen}
			f_{\widetilde{\gamma}}^\tau(k)=
			\begin{cases}
				f_\gamma^\tau(k)+\sqrt{(f_\gamma^\tau(k))^2-1}, &\text{provided }\;  f_\gamma^\tau(k)<-1 ,
				\\[2pt]
				f_\gamma^\tau(k)-\sqrt{(f_\gamma^\tau(k))^2-1}, &\text{provided }\;   f_\gamma^\tau(k)>1.
			\end{cases}
		\end{equation}
	\end{theorem}
	\begin{proof}
		Let $\psi\in\L(\Gamma_0)$ be the eigenfunction corresponding to $E$.
		We denote by $e_j$, $e_j^+ $, $e_j^- $ the edges of $\Gamma_0$
		connecting the vertices $x_{j-1}$,
		$x_j^+$ and $x_j^-$ with $x_j$, respectively (see Fig.~\ref{UnitCell-transfer}).
		On each edge $e_j$ we introduce the natural coordinate $x\in [x_{j-1},x_j]$, and similarly,
		on the edges $e_j^-$ and $e_j^+$ we introduce the natural coordinate
		$x\in [x_{j}^-,x_j]$ and $x\in [x_{j},x_j^+]$, respectively.
		We denote by $\psi_{j}(x)$ and $\psi_j^\pm(x)$
		the restrictions of $\psi(x)$ to $e_j$ and $e_j^\pm$, respectively.
		
		Since $\H\gaga^{\theta_2}$ acts as negative Laplacian on the edges, we have
		\begin{align}\label{PerAnsatz,xj}
			\begin{matrix}
				\psi_{j}(x)=\alpha_{j}\cdot \e^{ik(x-x_j)}+\beta_j\cdot\e^{-ik(x-x_j)},&x\in[x_{j-1},x_j],
				\\[3pt]
				\psi_{j}^+(x)=\alpha_j^+\cdot \e^{ik(x-x_j)}+\beta_j^+\cdot\e^{-ik(x-x_j)},& x\in[x_{j},x_j^+],
				\\[3pt]
				\psi_{j}^-(x)=\alpha_j^-\cdot \e^{ik(x-x_j)}+\beta_j^-\cdot\e^{-ik(x-x_j)},& x\in[x_j^-,x_j],
			\end{matrix}
		\end{align}
		with the coefficients to be determined, and similarly
		$$
		\psi_{j+1}(x)=\alpha_{j+1}\cdot \e^{ik(x-x_{j+1})}+
		\beta_{j+1}\cdot\e^{-ik(x-x_{j+1})},\quad x\in[x_{j},x_{j+1}].
		$$
		By assumption $\psi$ satisfies the $\delta$-conditions at the vertices $x_j$, which, in particular, implies
		its continuity. We denote by $\nu_j$ the value of $\psi$ at $x_j$, that is, one has
		\begin{align*}
			\nu_j&\ceq\psi_{j}(x_j)=\psi_{j+1}(x_j)=\psi_{j}^+(x_j)=\psi_{j}^-(x_j),
		\end{align*}
		and, furthermore, the relation
		\begin{align*}
			\gamma\nu_j=-\psi_j'(x_j)+\psi_{j+1}'(x_j)-{\psi_j^-}'(x_j)+{\psi_j^+}'(x_j).
		\end{align*}
	Note that $\psi$ is supposed to be square integrable, which requires
		\begin{equation}\label{L2}
		\sum_{j\in\mathbb{Z}}|\nu_j|^2<\infty.
	\end{equation}
		The above equalities imply the following relations between the parameters $\alpha_j$, $\al_j^\pm$,
		$\be_j$, $\be_j^\pm$\,:
		\begin{align}
			\nu_j&=
			\alpha_{j} +\be_{j}=
		 \alpha_{j+1}\;\e^{-ika} +\be_{j+1}\;\e^{ika}=
			\alpha_{j}^{+}+\be_{j}^{+}  =
			\alpha_{j}^{-}+\be_{j}^{-} ,
			\label{delta,xj}
			\\[2pt]
			\gamma \nu_j&=
			ik (
			-\alpha_j+\be_j+\alpha_{j+1}\;\e^{-ika}-\be_{j+1}\;\e^{ika}
			-\alpha_j^-+\be_j^{-}
			+\alpha_j^+-\be_j^+
			).
			\label{delta,der,xj}
		\end{align}
		We also have
		\begin{equation}\label{psi,xj+-1}
			\nu_{j-1}=\al_j\e^{-ika}+\be_j\e^{ika},\quad
			\nu_{j+1}=\alpha_{j+1}+\be_{j+1}.
		\end{equation}
		At the degree-one vertices of $\Gamma_0$ we have the quasi-periodic conditions
		$\psi_j^+(x_{j}^{+})=\e^{i \theta_{2} } \psi_j^-(x_{j}^{-})$ and
		$(\psi_j^+)'(x_{j}^{+})=\e^{i \theta_{2} } (\psi_j^-)'(x_{j}^{-})$,
		whence
		\begin{equation} \label{floq,xj}
			\begin{split}
				\e^{ikb/2}\alpha_j^{+}+\e^{-ikb/2}\be_j^+ =
				\e^{i \theta_{2} }\left(\e^{-ikb/2}\alpha_{j}^{-}+\e^{ikb/2}\be_{j}^{-}\right)    ,  \\[2pt]
				\e^{ikb/2}\alpha_j^{+}-\e^{-ikb/2}\be_j^+ =
				\e^{i \theta_{2} }\left(\e^{-ikb/2}\alpha_{j}^{-}-\e^{ikb/2}\be_{j}^{-}\right)   .
			\end{split}
		\end{equation}
		It follows from \eqref{delta,xj} that
		$\beta_{j+1}\e^{ika}-\beta_{j}=-\al_{j+1}\e^{-ika}+\al_j$,
		$\beta_{j}^+-\beta_{j}^-=-\al_{j}^++\al_j^-$. Substituting these
		equalities into \eqref{delta,der,xj}, we get
		\begin{equation}\label{delta,der,xj,sub}
			\gamma\nu_j=2ik\left( -\alpha_{j}+\alpha_{j+1}\e^{-ika}-\alpha_{j}^- +\alpha_{j}^+  \right).
		\end{equation}
		In a similar way, expressing $\alpha_{i}^{-}$ in terms of $\alpha_{i}^{+}$ and $\psi_{j}$ from \eqref{delta,xj}, and substituting them into \eqref{psi,xj+-1} and \eqref{floq,xj}, we arrive at
		\begin{equation}\label{alpha13+}
			\alpha_j=\frac{\nu_{j-1}-\nu_{j}\,e^{i ka}}{2\sin ka}   i ,
		\end{equation}
		and
		\begin{equation}\label{alpha24+}
			\alpha_{j}^{+}=\frac{ e^{i \left(kb+\theta _2\right)}-1}{ e^{2 i kb}-1} \;\nu_j,      \qquad   \alpha_{j}^{-}=\frac{e^{i kb} \left(e^{ikb}-e^{-i \theta _2}\right)}{ e^{2 i kb}-1}\;\nu_j    ,
		\end{equation}
		respectively. Substituting then \eqref{alpha13+} and \eqref{alpha24+} into \eqref{delta,der,xj,sub}, we obtain the following three-term recursion relation for the wave function values at the neighboring vertices,
		\begin{equation}\label{3terms,gen}
			\nu _{j+1}-2f_\ga^\tau(k) \nu _j+\nu _{j-1} =0 ,\quad j\not =0.
		\end{equation}
		where $\tau=\cos \theta_2 $. Similarly, for $j=0$ we get
		\begin{equation}\label{3terms,gen0}
			\nu_{1}-2f_{\wt \ga}^\tau(k) \nu _0+\nu _{-1} =0 .
		\end{equation}
		Next we introduce the vector $X_j=(\nu_{j},\nu_{j-1})^T$, $j\in\Z$.
		Then, due to \eqref{3terms,gen}, we get
		$$
		X_{j+1}=M X_j,\quad j\ge 1,\qquad
		X_{j-1}=M^{-1} X_{j},\quad j\le 0,
		$$
		where the transfer matrix $M$ is given by
		$$
		M=
		\left(
		\begin{matrix}
			2f_\ga^\tau(k) & -1
			\\
			1 & 0
		\end{matrix}
		\right)
		$$
		Recall that $|f^\tau_\gamma(k)|>1$ for $E\in (0,\infty)\setminus\sigma(\H_\ga^{\theta_2})$.
		The  eigenvalues $\Lambda_{\pm}(k)$ of $M$ are given by
		\begin{equation}\label{,lamb,j,sol}
			\Lambda_{\pm}(k) =f^\tau_\gamma(k)\pm\sqrt{ f^\tau_\gamma(k)^2-1}  \,  ,
		\end{equation}
		We denote by $Y_\pm$ the corresponding eigenvectors:
		$$
		Y_\pm=
		\left(
		\begin{matrix}
			\Lambda_\pm\\1
		\end{matrix}
		\right).
		$$
		The vectors $Y_+$ and $Y_-$ form a basis in $\C^2$, thus
		we can represent the vector $X_1$ as follows,
		$$
		X_1=c_- Y_- + c_+ Y_+,\text{ where }c_\pm\in\C,\ c_-^2+c_+^2\not=0,
		$$
		whence
		$$
		X_j=c_- (\Lambda_-)^{j-1} Y_- + c_+ (\Lambda_+)^{j-1} Y_+,\quad j\ge 1.
		$$
		If $f^\tau_\ga(k)>1$ (respectively, $f^\tau_\ga(k)<-1$), then
		$0<\Lambda_-<1$ and $\Lambda_+>1$
		(respectively, $-1<\Lambda_+<0$ and $\Lambda_-<-1$). Hence, in order to fulfil \eqref{L2}, we
		require $c_+=0$ (respectively, $c_-=0$), and consequently,
		$X_j=\Lambda_-^j X_1$ (respectively, $X_j=\Lambda_+^j X_1$).
		Without loss of generality we may assume that $c_-=1$ (respectively, $c_+=1$),
		whence
		\begin{gather}
			\nu_j=\Lambda_-^j\quad \text{(respectively, $\nu_j=\Lambda_+^j$)},\quad j\ge 0.
		\end{gather}
		Similarly,
		\begin{gather}
			\nu_j=\Lambda_+^j\quad \text{(respectively, $\nu_j=\Lambda_-^j$)},\quad j\le 0.
		\end{gather}
		(note that $\Lambda_+\Lambda_-=1$). Finally, taking into account \eqref{3terms,gen0}, we get
		\begin{gather}
			\Lambda^+ - f_{\widetilde{\gamma}}^\tau(k)=0 \text{\quad(respectively, $\Lambda^- -  f_{\widetilde{\gamma}}^\tau(k)=0$)}
		\end{gather}
		This concludes the proof.
		\end{proof}
Similarly, we obtain the following statement about the perturbation-generated negative eigenvalues inside the gaps of $\sigma(\H_\ga^{\theta_2})$:

\begin{corollary}
    \label{cor:per:comb;neg}
	Let $E\in (-\infty,0)\setminus \sigma(\H_\ga^{\theta_2})$.
	We set $\kappa=|E|^{1/2}$, $\tau=\cos\theta_2$. Then $E\in\sigma_{\rm disc}(\H\gaga^{\theta_2})$ iff
	\begin{equation}\label{sol,gen,neg}
		{ \wh f_{\widetilde{\gamma}}^\tau}(\kappa)=
		\begin{cases}
			{ \wh f_\gamma^\tau}(\kappa)+\sqrt{({ \wh f_\gamma^\tau}(\kappa))^2-1}, &\text{provided }\;  { \wh f_\gamma^\tau}(\kappa)<-1 ,
			\\[2pt]
			{ \wh f_\gamma^\tau}(\kappa)-\sqrt{({ \wh f}_\gamma^\tau(\kappa))^2-1}, &\text{provided }\; { \wh f_\gamma^\tau}(\kappa)>1.
		\end{cases}
	\end{equation}
\end{corollary}
In Figs.~\ref{Ex1,a1b3gam4gamt1} and \ref{Ex2,a2b2gam1gamt3}, the pattern of $\sigma_{\rm disc}(\H\gaga^{\theta_2})$ is illustrated for particular choices of parameters; for simplicity we plot the fiber-operator spectrum for a (dense enough) discrete array of $\theta_2$ values.

	\begin{figure}[h]
		\centering
		\includegraphics[scale=0.828]{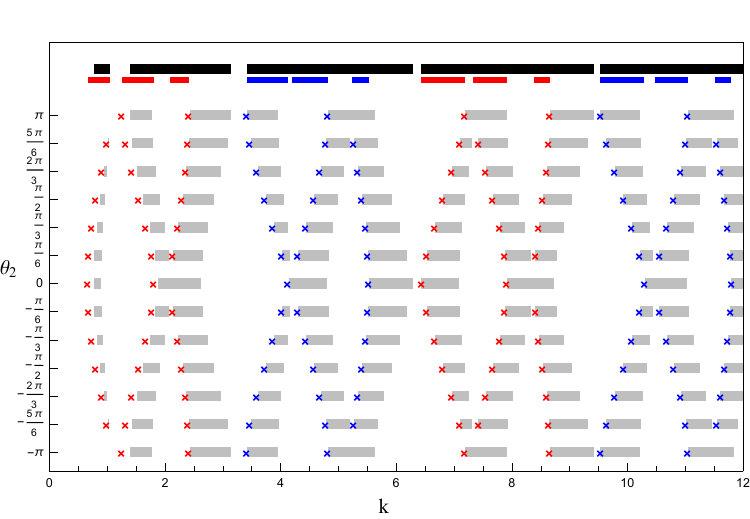}
		\caption{The spectral picture for the lattice with $a=1$, $b=3$, $\gamma=4$ and $\widetilde{\gamma}=1$. The cross marks illustrate $\sigma_{\rm disc}(\H\gaga^{\theta_2})$ for the indicated values of $\theta_2$ (the blue and red colours correspond respectively to the upper and lower conditions in \eqref{sol,gen}); the union of these discrete eigenvalues over $\theta_2\in[-\pi,\pi]$ is illustrated by stripes in blue (determined by \eqref{Pert,BandCon1}) and red (determined by \eqref{Pert,BandCon2}) on the top of the picture. The gray stripes represent $\sigma(\H_\ga^{\theta_2})$ for the same values of $\theta_2$, determined by \eqref{sc,per,gen}; the union of these bands over $\theta_2\in[-\pi,\pi]$ yields the spectral bands of the operator $\H_\ga$ which is illustrated by black stripes on the top (determined by the condition \eqref{per,band,con}). Spectral bands of the perturbed operator $\H\gaga$ will be then the union of the black, red and blue stripes on the top.}
		\label{Ex1,a1b3gam4gamt1}
	\end{figure}

	\begin{figure}[h]
		\centering
		\includegraphics[scale=0.828]{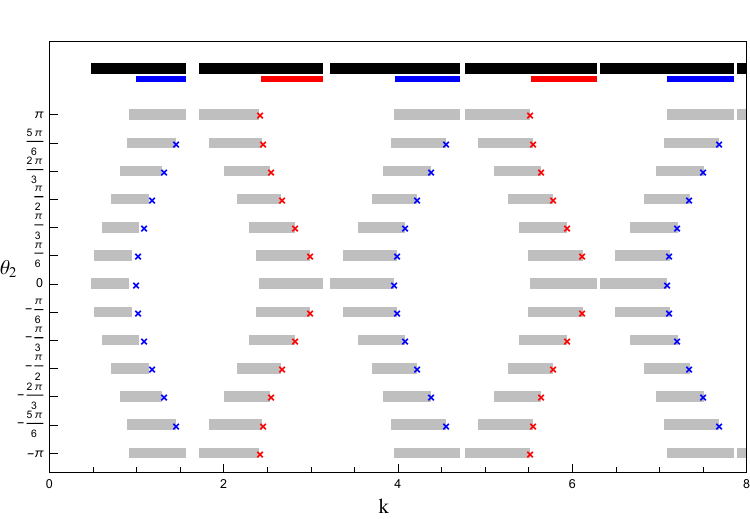}
		\caption{The same as above with the parameters $\gamma=1$ and $\widetilde{\gamma}=3$ for the square lattice, $a=b=2\,$; note that the blue and red crosses are, respectively, missing at $\pm\pi$ and zero since the new eigenvalues there merge the band edges. This example corresponds to the case~\ref{th4} of Theorem~\ref{thmGen} below in which the spectrum of $\H_\ga$ remains unchanged under the perturbation.}
		\label{Ex2,a2b2gam1gamt3}
	\end{figure}

	\newpage

         Theorem~\ref{thm:per:comb} and Corollary~\ref{cor:per:comb;neg} will be used later on to detect the new spectral bands resulting from the perturbation (see Proposion~\ref{proposition4}).
        To handle the cases when no new bands appear we make the following claim about conditions characterizing the set of new spectral bands resulting from the perturbation
        (as Figs.~\ref{Ex1,a1b3gam4gamt1} and \ref{Ex2,a2b2gam1gamt3} show, they may or may not be embedded in the original bands of $\sigma(\H_\ga)$).

	\begin{proposition}\label{cor:explicitBands}
	A positive number $E=k^2$ belongs to the set of spectral bands $$S\gaga\coloneqq\bigcup_{\theta_2\in [-\pi,\pi]}\,\sigma_{\rm disc}(\H\gaga^{\theta_2})$$ iff the number $k$ satisfies either the band condition
		\begin{equation}\label{Pert,BandCon1}
			  k\in \left\{ \,	\mathcal{BC}_k^+  \quad \cap \quad  \mathcal{G}(k)>0  \right\},
		\end{equation}
		or the band condition
		\begin{equation}\label{Pert,BandCon2}
			 k\in \left\{ \,	\mathcal{BC}_k^-  \quad \cap \quad  \mathcal{G}(k)<0  \right\},
		\end{equation}
		where $ \mathcal{G}(k):=(\widetilde{\gamma}-\gamma )\,\frac{ \sin k a}{2 k }$ and
		\begin{equation}\label{BCpmk}
			\mathcal{BC}_k^{\pm}:\quad \left\lvert   \frac{ \sin k(a+b)}{\sin k b}+\frac{\gamma\, \sin ka}{2 k}\pm\sqrt{\left( (\widetilde{\gamma}-\gamma )\,	\frac{ \sin ka}{2 k} \right)^2+1} \right\rvert   \leq   \left\lvert \frac{\sin ka }{\sin kb}  \right\rvert	.
		\end{equation}
		Similarly, by replacing $k$ by $i \kappa$ with $\kappa>0$ in Eqs. \eqref{Pert,BandCon1} and \eqref{Pert,BandCon2}, a negative number $E=-\kappa^2$ belongs to $S\gaga$ iff $\kappa$ satisfies either the band condition
		\begin{equation}\label{Pert,BandCon1,neg}
			 \kappa \in \left\{ \ \wh{\mathcal{BC}}_\kappa^{+}  \quad \cap \quad  \wh{\mathcal{G}}(\kappa)>0  \right\},
		\end{equation}
		or the band condition
		\begin{equation}\label{Pert,BandCon2,neg}
			 	\kappa \in \left\{ \ \wh{\mathcal{BC}}_\kappa^{-}  \quad \cap \quad \wh{\mathcal{G}}(\kappa)<0  \right\},
		\end{equation}
		where $ \wh{\mathcal{G}}(\kappa):=(\widetilde{\gamma}-\gamma )\,\frac{ \sinh \kappa a}{2 \kappa }$ and
		\begin{equation}\label{BCpmkappa}
			 \wh{\mathcal{BC}}_\kappa^{\pm}:\quad 	\left\lvert   \frac{ \sinh \kappa (a+b)}{\sinh \kappa  b}+\frac{\gamma\, \sinh \kappa a}{2 \kappa }\pm\sqrt{\left( (\widetilde{\gamma}-\gamma )\,	\frac{ \sinh \kappa a}{2 \kappa } \right)^2+1} \right\rvert   \leq    \frac{\sinh \kappa a }{\sinh \kappa b}  	.
		\end{equation}
	\end{proposition}
	\begin{proof}
        Let $E\in S_{\gamma,\wt\gamma}\cap(0,\infty)$.
        Then by Theorem~\ref{thm:per:comb}, $k=E^{1/2}$ satisfies \eqref{sol,gen} with some $\tau\in [-1,1]$.
        Let us assume, for example, that the first condition in \eqref{sol,gen} holds true. Substituting the explicit forms of the functions $f_{\gamma}^\tau(k)$ and $f_{\widetilde{\gamma}}^\tau(k)$ from \eqref{sc,alter,adiffb} into it, we can rewrite it after simple manipulations in the form
		\begin{equation}\label{simp,lam=lamp}
			( \widetilde{\gamma}-\gamma )\,	\frac{ \sin ka}{2 k}=\sqrt{  f^\tau_\gamma(k)^2-1},
		\end{equation}
		which requires $( \widetilde{\gamma}-\gamma )\,	\frac{ \sin ka}{2 k}>0$. Next, squaring both sides of \eqref{simp,lam=lamp}, we get
		\begin{equation}\label{squ,lam=lamp}
			f^\tau_\gamma(k)^2=\left(( \widetilde{\gamma}-\gamma )\,	\frac{ \sin ka}{2 k}\right)^2+1,
		\end{equation}
		which yields
		\begin{equation}\label{squRoot12,lam=lamp}
			f^\tau_\gamma(k)=\pm \sqrt{\left(( \widetilde{\gamma}-\gamma )\,	\frac{ \sin ka}{2 k}\right)^2+1}\,;
		\end{equation}
		the upper sign here contradicts the initial constraint $f^\tau_\gamma(k)<-1$ associated with the first condition in \eqref{sol,gen}. Substituting then the explicit form of $f^\tau_\gamma(k)$ into \eqref{squRoot12,lam=lamp} with the lower sign, we arrive after simple manipulations to the condition of the form
		\begin{equation}\label{equ,lam=lamPos}
			\frac{\sin ka }{\sin kb}\,\tau=\frac{ \sin k(a+b)}{\sin k b}+\frac{\gamma\, \sin ka}{2 k}+\sqrt{\left( (\widetilde{\gamma}-\gamma )\,	\frac{ \sin ka}{2 k} \right)^2+1},
		\end{equation}
		which, as mentioned above, has to be accompanied by the constraint $(\widetilde{\gamma}-\gamma )\,	\frac{ \sin ka}{2 k}>0$.
        Since $|\tau|\leq1$,
        \eqref{equ,lam=lamPos} implies $\mathcal{BC}_k^{+}$.
        Mimicking the argument for the second condition in \eqref{sol,gen}, in which case we have $f^\tau_\gamma(k)>1$, we arrive at
		\begin{equation}\label{equ,lam=lamNeg}
			\frac{\sin ka }{\sin kb}\,\tau=\frac{ \sin k(a+b)}{\sin k b}+\frac{\gamma\, \sin ka}{2 k}-\sqrt{\left( (\widetilde{\gamma}-\gamma )\,	\frac{ \sin ka}{2 k} \right)^2+1},
		\end{equation}
		accompanied by the constraint $(\widetilde{\gamma}-\gamma )\,\frac{ \sin ka}{2 k}<0$; evidently, \eqref{equ,lam=lamNeg} implies  $\mathcal{BC}_k^{+}$.
        Applying the above arguments   backwards, one can easily get the
        converse statement: \eqref{Pert,BandCon1,neg} (respectively, \eqref{Pert,BandCon2,neg}) implies the fulfillment of the first equality in \eqref{sol,gen,neg} (respectively, the second equality in \eqref{sol,gen,neg}) with some $\tau\in [-1,1]$.
             As for the negative spectrum, one has to replace simply $k$ by $i \kappa$ with $\kappa>0$ in \eqref{equ,lam=lamPos} and \eqref{equ,lam=lamNeg} as well as in the respective constraints.
	\end{proof}

	Figs.~\ref{threefigsRect}--\ref{threefigssquare} show the spectral bands of $\sigma(\H_\ga)$ (marked in gray) together with the \emph{newly added} spectral bands resulting from the perturbation; we show only the energies $E\in S\gaga$ \emph{inside} the spectral gaps of $\sigma(\H_\ga)$. We plot them as functions of $\widetilde{\gamma}$ both in commensurate and incommensurate cases; in the latter we choose the golden-mean as the `worst' irrational for the ratio of the edge lengths $a$ and $b$. To provide an insight, we use two different colors: the blue color for the new spectral bands determined by Eqs.~\eqref{Pert,BandCon1} and \eqref{Pert,BandCon1,neg}, and the red for those by Eqs.~\eqref{Pert,BandCon2} and \eqref{Pert,BandCon2,neg}.

	\begin{figure}[h]
		\centering
		\subfloat[\centering  $\gamma=10$,]{{\includegraphics[width=1.8 in]{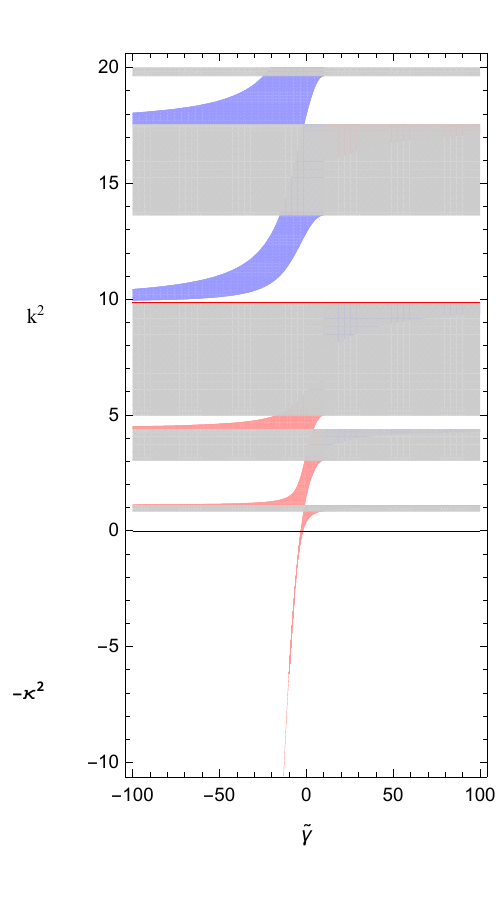} }\label{Sh2,Pert,gam+10,a1,b3,vs,gammt,perOver}}%
		\quad
		\subfloat[\centering $\gamma=-4$,]{{\includegraphics[width=1.8 in]{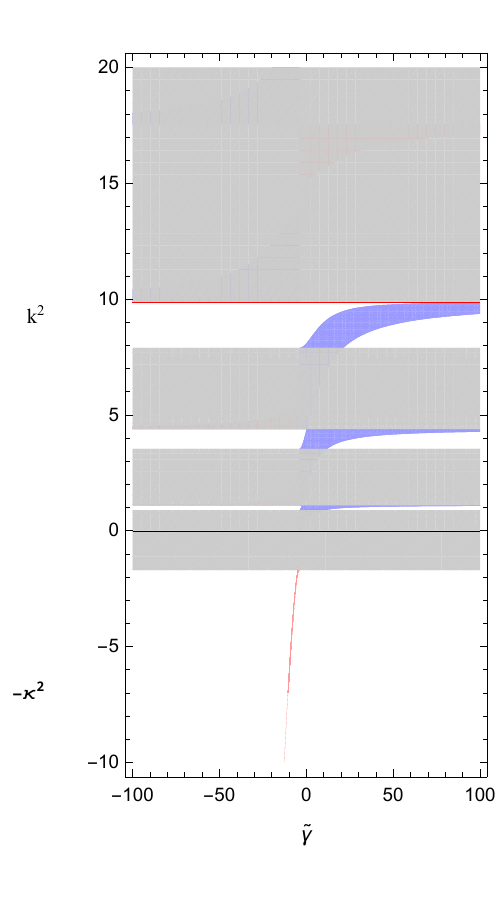} }\label{Sh1,Pert,gam-4,a1,b3,vs,gammt,PerOver}}%
		\quad
		\subfloat[\centering $\gamma=-10$.]{{\includegraphics[width=1.8 in]{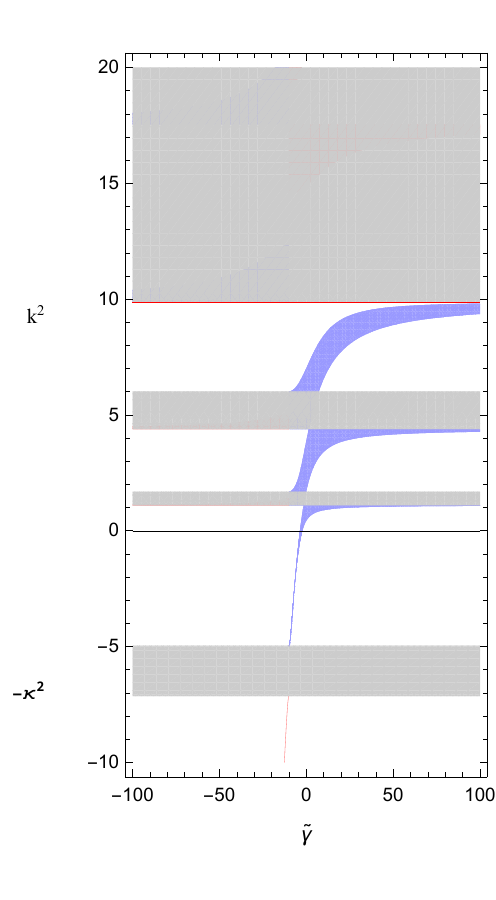} }\label{Sh3,Pert,gam-10,a1,b3,vs,gammt,PreOver}}%
		\caption{Spectrum of the perturbed rectangular lattice, $\sigma(\H\gaga)$, as a function of $\widetilde{\gamma}$\, for $a=1$, $b=3$ and three different values of  $\gamma$. The gray area represents the spectral bands of $\sigma(\H_{\gamma})$. The blue area (referring to conditions \eqref{Pert,BandCon1} and \eqref{Pert,BandCon1,neg}) and the red area (referring to conditions \eqref{Pert,BandCon2} and \eqref{Pert,BandCon2,neg}) represent the new spectral bands resulting from the perturbation, i.e. energies $E\in S\gaga$ appearing inside the gaps of $\sigma(\H_\ga)$. The red line corresponds to the flat bands $(\frac{n\pi}{a})^2=(\frac{m\pi}{b})^2$ with $n=1$ and $m=3$).}
		\label{threefigsRect}
	\end{figure}

\newpage

\begin{figure}[h]
	\centering
	\subfloat[\centering  $\gamma=30$,]{{\includegraphics[width=1.8 in]{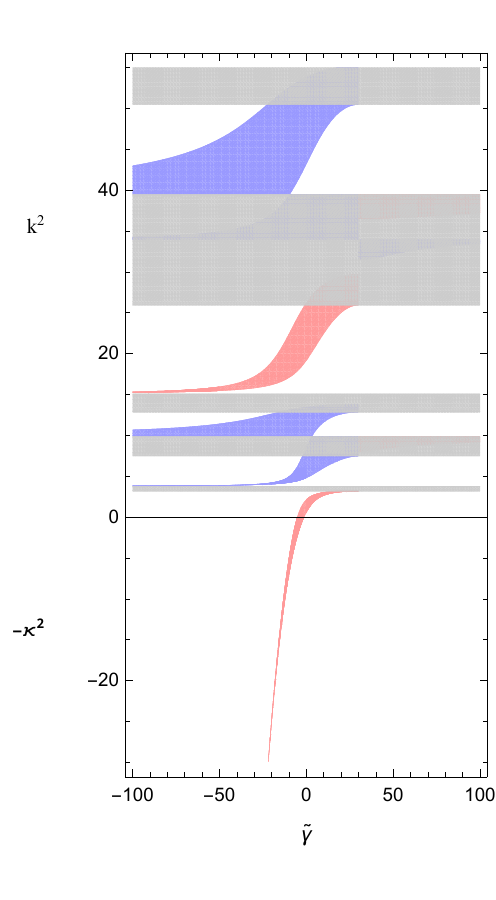} }\label{gam+30,abGolden}}%
	\quad
	\subfloat[\centering $\gamma=-5$,]{{\includegraphics[width=1.8 in]{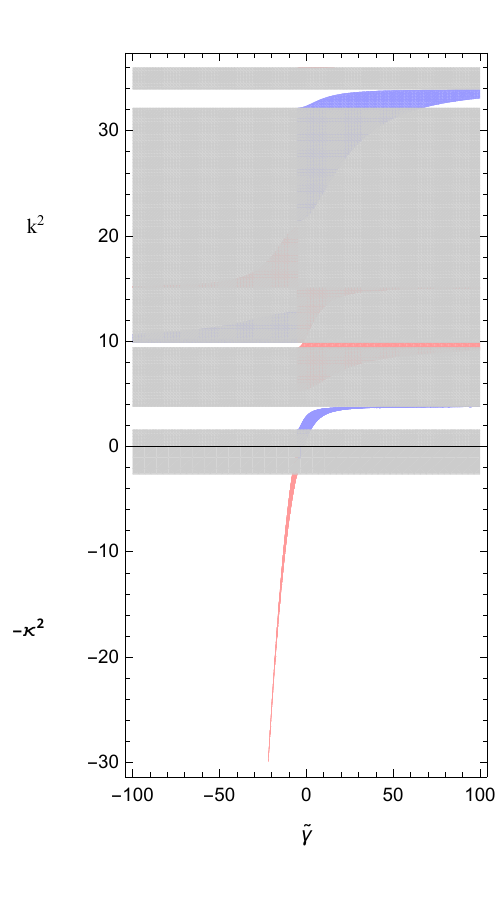} }\label{gam-5,abGolden}}%
	\quad
	\subfloat[\centering $\gamma=-25$.]{{\includegraphics[width=1.8 in]{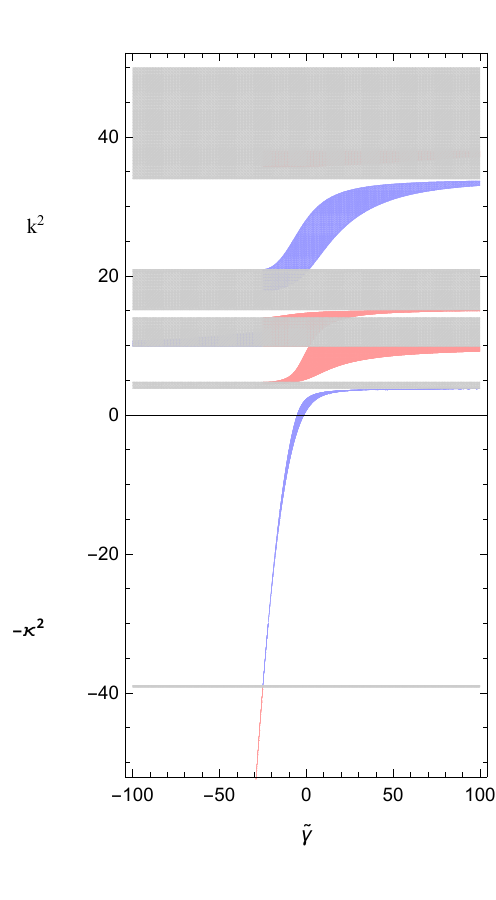} }\label{gam-25,abGolden}}%
	\caption{The same as above for the golden-mean ratio of the edges,  $a=\frac{\sqrt{5}+1}{2}$ and $b=1$, with three different values of $\gamma$ for which the number of open gaps is infinite, cf. Remark~\ref{rem:Bethe-Sommerfeld}.}
	\label{threefigsRectGolden}
\end{figure}

\begin{figure}[h]
	\centering
	\subfloat[\centering $\gamma=-15$,]{{\includegraphics[width=1.8 in]{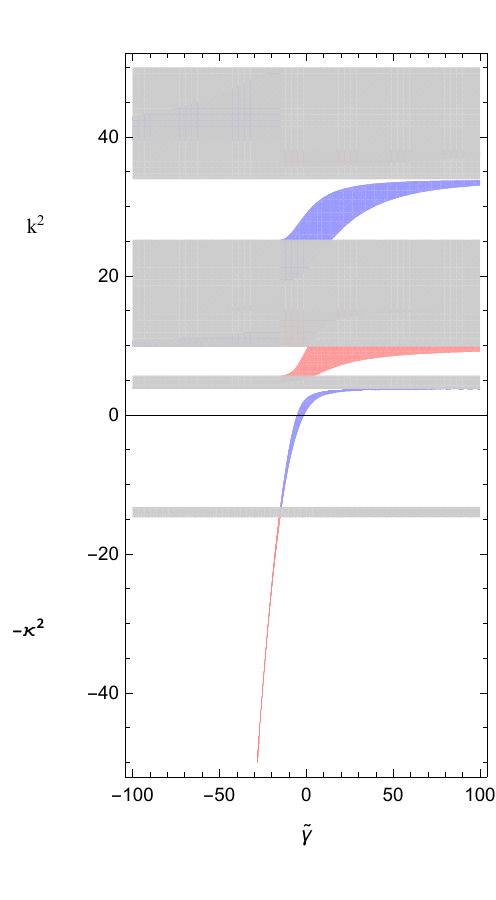} }\label{badGamma,minus15}}%
	\quad
	\subfloat[\centering $\gamma=-2.72$,]{{\includegraphics[width=1.8 in]{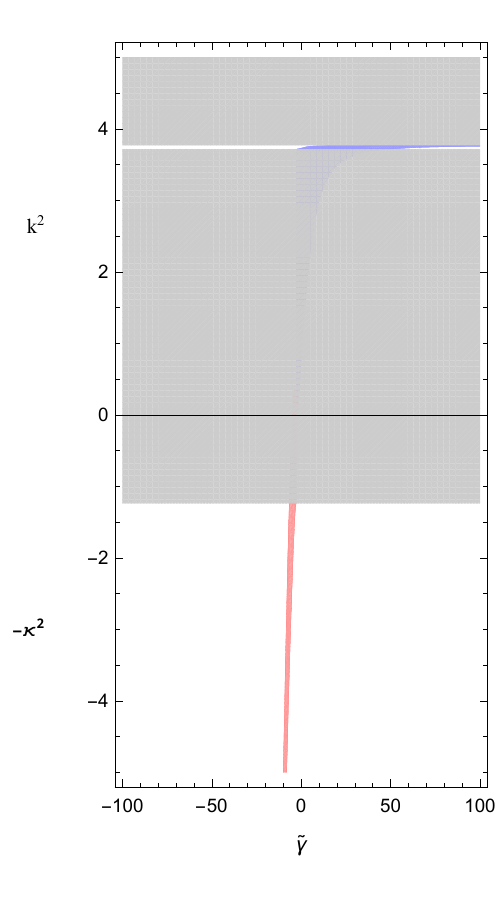} }\label{badGamma,minus2point72}}%
	\quad
	\subfloat[\centering $\gamma=-2$.]{{\includegraphics[width=1.8 in]{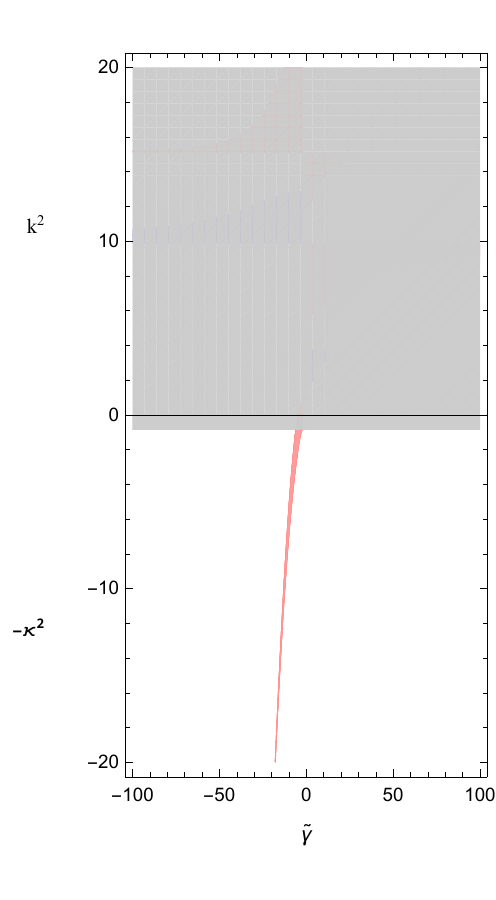} }\label{badGamma,minus2}}%
	\caption{The same as above for the golden-mean ratio of the edges,  $a=\frac{\sqrt{5}+1}{2}$ and $b=1$, with three particular values of $\gamma$ illustrating the situations in which the number of original open gaps is infinite, non-zero finite (a single one around $k^2\approx 3.75$ in this case), and zero (excluding the semi-infinite gap); cf. Remark~\ref{rem:Bethe-Sommerfeld}. Also, since it is not well visible at the chosen scale, let us mention that the upper edge curve of the new blue band, appearing inside the single gap of $\sigma(\H\gaga)$ in the plot \textbf{(b)}, does not touch the lower edge of the second original band which corresponds to $k=\frac{n\pi}a$ with $n=1$, or $k^2\approx 3.77$; in other words, the mentioned gap does not get completely closed in this case.}
	\label{threefigsRectGolden3}
\end{figure}

\newpage

	\begin{figure}[h]
		\centering
		\subfloat[\centering  $\gamma=10$,]{{\includegraphics[width=1.8 in]{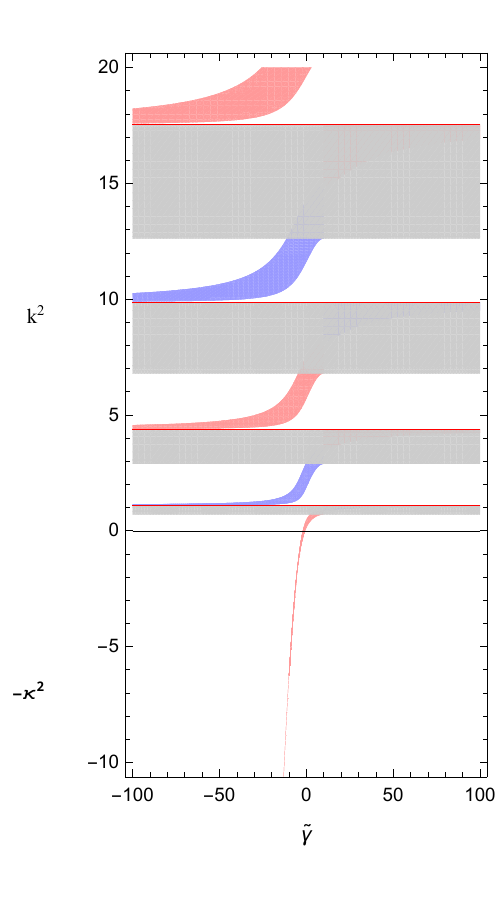} }\label{SQ,Sh2,Pert,gam+10,a3b3,vs,gammt,perOver}}%
		\quad
		\subfloat[\centering $\gamma=-4$,]{{\includegraphics[width=1.8 in]{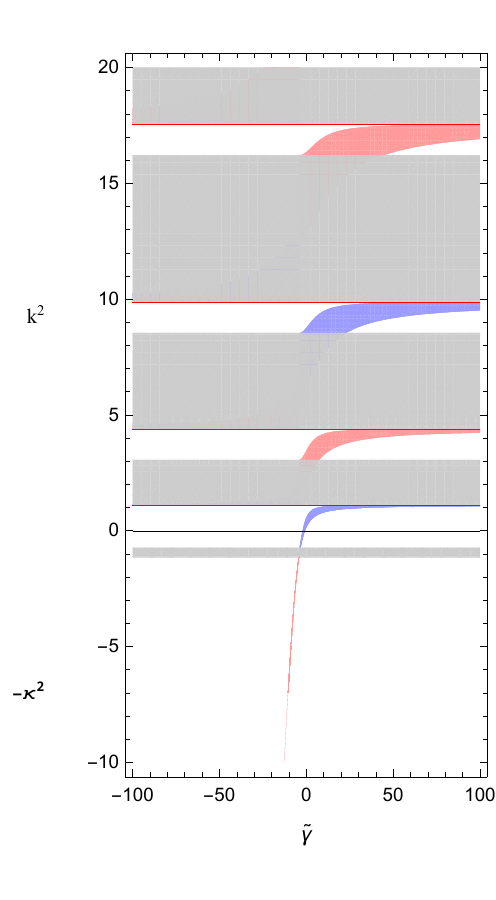} }\label{SQ,Sh1,Pert,gam-4,a3b3,vs,gammt,PerOver}}%
		\quad
		\subfloat[\centering $\gamma=-2$.]{{\includegraphics[width=1.8 in]{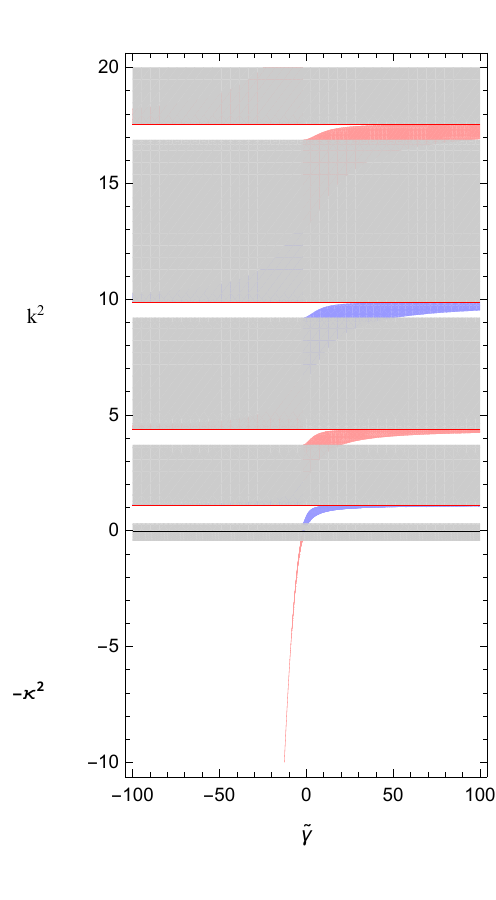} }\label{SQ,Sh3,Pert,gam-2,a3b3,vs,gammt,PreOver}}%
		\caption{The same as above for the square lattice case, $a=b=3$. The red lines refer to the flat bands $(\frac{n\pi}{a})^2$ with $n=1,2,3,4$. Again, since it is not well visible at the chosen scale, note that none of the edge curves of the new bands touches those edges of original bands corresponding to the mentioned flat bands; cf. Remark~\ref{remarkGapCloseSq}.}%
		\label{threefigssquare}%
	\end{figure}

		In what follows, Propositions~\ref{proposition1}--\ref{proposition3} deal with the situations in which the perturbation \emph{does not} give rise to new spectral bands inside the spectral gaps of $\sigma(\H_\ga)$. In contrast, Proposition~\ref{proposition4} refers to situations when the perturbation \emph{does} induce to a new band inside \emph{each} spectral gap of $\sigma(\H_\ga)$. All these results will be summarized in Theorem~\ref{thmGen} at the end of this section.

	\begin{proposition}\label{proposition1}
		\textit{If $\widetilde{\gamma}>\gamma>0$, the spectral gaps of $\sigma(\H_\ga)$ remain open and unchanged under the perturbation, cf. Figs.~\ref{threefigsRect}, \ref{threefigsRectGolden} and \ref{threefigssquare}. }
	\end{proposition}
	\begin{proof}
		Let us start with the \emph{positive} spectrum. Such a behavior can be established by inspecting whether it is possible that an eigenvalue $E=k^2\in S\gaga$, i.e. referring to $k$ satisfying the conditions \eqref{Pert,BandCon1} and \eqref{Pert,BandCon2}, may appear inside the spectral gaps of $\sigma(\H_\ga)$ or not. First of all, from \eqref{per,band,con} we infer that a number $k^2$ belongs to the spectral gaps of $\sigma(\H_\ga)$ if and only if $k$ satisfies either the gap condition
		\begin{equation}\label{per,gap,con1}
			 	\frac{\gamma \, \sin ka}{2 k}+\frac{ \sin k(a+b)}{\sin k b}  >  1+\left\lvert \frac{\sin ka }{\sin kb}\right\rvert ,
		\end{equation}
	or
		\begin{equation}\label{per,gap,con2}
		 \frac{\gamma \, \sin ka}{2 k}+\frac{ \sin k(a+b)}{\sin k b}  <  -1-\left\lvert \frac{\sin ka }{\sin kb}\right\rvert .
		\end{equation}
	Next, comparing these inequalities with the conditions ensuring that $k$ belongs to the positive part of the set $S\gaga$, i.e. Eqs.~\eqref{Pert,BandCon1} and \eqref{Pert,BandCon2}, one immediately finds that the gap condition \eqref{per,gap,con1} (respectively, \eqref{per,gap,con2}) cannot be fulfilled simultaneously with $\mathcal{BC}_k^+$ in \eqref{Pert,BandCon1} (respectively, $\mathcal{BC}_k^-$ in \eqref{Pert,BandCon2}) in view of the fact that the square root term in $\mathcal{BC}_k^{\pm}$ in \eqref{BCpmk} is always positive. Hence it remains to check whether the conditions \eqref{per,gap,con1} (respectively, \eqref{per,gap,con2}) and  \eqref{Pert,BandCon2}  (respectively, \eqref{Pert,BandCon1}) can be fulfilled simultaneously or not. To this aim, we note first that in view of the inequality
	$$
	\left\lvert \frac{ \sin k(a+b)}{\sin k b} \right\rvert=\left\lvert \cos ka +\cos kb\; \frac{\sin ka }{\sin kb} \right\rvert \leq  1+\left\lvert \frac{\sin ka }{\sin kb}\right\rvert,
	$$
	the conditions \eqref{per,gap,con1} and \eqref{per,gap,con2} imply $\frac{\gamma \, \sin ka}{2 k}>0$ and $\frac{\gamma \, \sin ka}{2 k}<0$, respectively. On the other hand, one easily checks that if  $\widetilde{\gamma}>\gamma>0$ (respectively, if $\widetilde{\gamma}<\gamma<0$), the necessary constraint $\mathcal{G}(k)<0$  in \eqref{Pert,BandCon2} (respectively, $\mathcal{G}(k)>0$ in \eqref{Pert,BandCon1}) implies  for $\frac{\gamma \, \sin ka}{2 k}<0$  (respectively,  $\frac{\gamma \, \sin ka}{2 k}>0$). This entails that the gap conditions \eqref{per,gap,con1}, \eqref{per,gap,con2}, cannot be satisfied simultaneously with the band conditions  \eqref{Pert,BandCon1}, \eqref{Pert,BandCon2} for $\widetilde{\gamma}>\gamma>0$, as well as for $\widetilde{\gamma}<\gamma<0$ where the latter case will be the object of the next proposition.

The proof for the \emph{negative} spectrum proceeds in a similar way, by inspecting the conditions determining the spectrum in the gaps of $\sigma(\H_\ga)$. From \eqref{per,neg,band,con}, a negative eigenvalue $-\kappa^2$ belongs to a spectral gap of $\sigma(\H_\ga)$ iff $\kappa$ satisfies either the gap condition
	\begin{equation}\label{per,neg,gap,con1}
		 	\frac{\gamma \, \sinh \kappa a}{2 \kappa}+  \frac{\sinh \kappa (a+b)}{\sinh \kappa b}<
		-1-\frac{\sinh \kappa a }{\sinh \kappa b}\;,
	\end{equation}
or
	\begin{equation}\label{per,neg,gap,con2}
		 	\frac{\gamma \, \sinh \kappa a}{2 \kappa}+  \frac{\sinh \kappa (a+b)}{\sinh \kappa b}>
		1+  \frac{\sinh \kappa a }{\sinh \kappa b} .
	\end{equation}
	Note that conditions \eqref{per,neg,gap,con1} and \eqref{per,neg,gap,con2} refer, respectively, to the spectral gap \emph{above} and \emph{below} $\inf\sigma(\H_\ga)$ where the latter is semi-infinite; recall also that the negative spectrum of $\H_\ga$ cannot have more than one  band. Next, mimicking the above argument, by comparing  the gap conditions \eqref{per,neg,gap,con1} and \eqref{per,neg,gap,con2} with those determining the negative part of $S\gaga$, i.e. \eqref{Pert,BandCon1,neg} and \eqref{Pert,BandCon2,neg}, we easily find that \eqref{per,neg,gap,con1} (respectively, \eqref{per,neg,gap,con2}) is incompatible with $\wh{\mathcal{BC}}_\kappa^-$ in  \eqref{Pert,BandCon2,neg} (respectively, with $\wh{\mathcal{BC}}_\kappa^+$ in  \eqref{Pert,BandCon1,neg}).
	
	It remains thus to check whether the conditions \eqref{per,neg,gap,con1} (respectively, \eqref{per,neg,gap,con2}) and \eqref{Pert,BandCon1,neg}  (respectively, \eqref{Pert,BandCon2,neg}) can be fulfilled simultaneously or not. To begin with, taking into account the definition of the function $ \wh{\mathcal{G}}(\kappa):=(\widetilde{\gamma}-\gamma )\,\frac{ \sinh \kappa a}{2 \kappa }$ and the inequality $\sinh x>0$ for $x>0$, we find that the conditions \eqref{Pert,BandCon1,neg} and  \eqref{Pert,BandCon2,neg} are fulfilled only for $\widetilde{\gamma}>\gamma$ and $\widetilde{\gamma}<\gamma$, respectively. Consequently, for $\widetilde{\gamma}>\gamma$ which is the case we are interested in here, the only possibility that an eigenvalue $-\kappa^2\in S\gaga$ belongs to a spectral gap of $\sigma(\H_\ga)$ is that $\kappa$ simultaneously satisfies the band condition \eqref{Pert,BandCon1,neg} and the gap condition \eqref{per,neg,gap,con1}. However, as mentioned above, \eqref{per,neg,gap,con1} refers to the gap \emph{above} $\inf\sigma(\H_{\gamma})$ which is strictly positive for $\gamma>0$, thus the claim follows.	

Let us note that the result could also be justified by manipulating the condition \eqref{per,neg,gap,con1} and rewriting it in the form
$$
\frac{\gamma \, \sinh \kappa a}{2 \kappa} <  -(\sinh \kappa a+\sinh \kappa b)\coth \frac{\kappa b}{2}\,,
$$
in which the function on the right-hand side is negative in view of that $\sinh x>0$ and $\coth x>0$ for $x>0$. This implies that \eqref{per,neg,gap,con1} holds only for $\gamma<0$ which is excluded by assumption; this confirms one more time the claim.
\end{proof}

\begin{proposition}\label{proposition2}
	\textit{If $\widetilde{\gamma}<\gamma<0$, the original spectral gaps above $\inf\sigma(\H_{\gamma})$ remain open and unchanged under the perturbation, cf. Figs.~\ref{threefigsRect}--\ref{threefigssquare}.   }
\end{proposition}
\begin{proof}
The claim for the \emph{positive} part of the spectrum has already been proved in Proposition~\ref{proposition1}. Concerning the \emph{negative} part, let us first recall that $\sigma(\H_{\gamma})$ contains a single negative band if $\gamma$ is negative, and that \eqref{per,neg,gap,con1} is the gap condition referring to the gap \emph{above} $\inf\sigma(\H_{\gamma})$. On the other hand, in Proposition~\ref{proposition1} we proved that \eqref{per,neg,gap,con1} can be fulfilled simultaneously \emph{only} with the band condition \eqref{Pert,BandCon1,neg} which is itself valid exclusively for $\widetilde{\gamma}>\gamma$; this yields the claim.
\end{proof}

\begin{proposition}\label{proposition3}
	\textit{If $\gamma<0$ and $\widetilde{\gamma}>\gamma$, the semi-infinite gap below $\inf\sigma(\H_{\gamma})$ remains open and unchanged under the perturbation, cf. Figs.~\ref{threefigsRect}--\ref{threefigssquare}  .}
\end{proposition}
\begin{proof}
Recall that \eqref{per,neg,gap,con2} is the gap condition referring to the semi-infinite gap \emph{below} $\inf\sigma(\H_{\gamma})$ which is strictly negative if $\gamma<0$. On the other hand, in Proposition~\ref{proposition1} we demonstrated that \eqref{per,neg,gap,con2} can be fulfilled simultaneously only with the band condition \eqref{Pert,BandCon2,neg} which is itself valid exclusively for $\widetilde{\gamma}<\gamma$; this proves the claim.
\end{proof}

Apart from the situations mentioned in Propositions~\ref{proposition1}--\ref{proposition3} in which the indicated spectral gaps of $\sigma(\H_\ga)$ remained unchanged under the perturbation, there is always a new spectral band of $S\gaga$ in each spectral gap of $\sigma(\H_\ga)$ which, as we will see below, may or may not touch the edges of the original bands of $\sigma(\H_\ga)$.

More precisely, inspecting Figs.~\ref{threefigsRect}--\ref{threefigssquare}, we observe that for $\widetilde{\gamma}\leq\gamma$, there is always a new spectral band of $S\gaga$ in the first gap below the threshold of continuous spectrum, moving from $\inf\sigma(\H_{\gamma})$ towards $-\infty$ as $\widetilde{\gamma}$ decreases. In contrast, in each original gap above $\inf\,\sigma(\H_{\gamma})$, there is always a new spectral band of $S\gaga$ in the whole range $\widetilde{\gamma}\ge \gamma$ (respectively, $\widetilde{\gamma}\leq\gamma$) for $\gamma<0$  (respectively, $\gamma>0$).

These observations will be proved in Proposition~\ref{proposition4} below; to proceed towards the proof, let us start with the following statements on the functions describing the endpoints of the spectral bands in $S\gaga$.

\begin{remark}\label{RemarkBords}
The conditions \eqref{equ,lam=lamPos} and \eqref{equ,lam=lamNeg}, characterizing the new positive spectral bands resulting from the perturbation, i.e. $k^2\in S\gaga\,$, can be rewritten in the new form
	\begin{align}\notag
		\widetilde{\gamma} &=\gamma\pm\frac{2k}{\sin ka}\sqrt{ \left( \frac{\gamma \, \sin ka}{2 k}+\frac{ \sin k(a+b)}{\sin k b}-\tau\, \frac{\sin ka }{\sin kb}\right)^2-1}\\\label{proof,lam=lamp,equiv}
        &=
        \gamma\pm\frac{2k}{\sin ka}\sqrt{ \left(f_\ga^\tau(k)\right)^2-1}
        =:g_{\pm}(\tau;k).
	\end{align}
Note that $g_{+}(\tau;k)$ and $g_{-}(\tau;k)$ in \eqref{proof,lam=lamp,equiv} are, respectively, associated with the constraints $f_\ga^\tau(k)<-1$ and $f_\ga^\tau(k)>1$, see Theorem~\ref{thm:per:comb} and Proposition~\ref{cor:explicitBands}.
Similarly, replacing $k$ by $i\kappa$ with $\kappa>0$ in \eqref{proof,lam=lamp,equiv},  we conclude that the negative bands of $S\gaga$ are characterized by
	\begin{align}\notag
		\widetilde{\gamma} &=\gamma\pm\frac{2\kappa}{\sinh \kappa a}\sqrt{ \left( \frac{\gamma \, \sinh \kappa a}{2 \kappa}+\frac{ \sinh \kappa(a+b)}{\sinh \kappa b}-\tau\, \frac{\sinh \kappa a }{\sinh \kappa b}\right)^2-1}
        \\\label{neg,gpm}
        &=\gamma\pm\frac{2\kappa}{\sinh \kappa a}\sqrt{\left(\wh f_\gamma^\tau(k)\right)^2-1}
        =:\wh g_{\pm}(\tau;\kappa)\, ,
	\end{align}
where, again, the upper and lower signs in the condition are, respectively, accompanied by the constraints ${\wh f_\ga^\tau}(\kappa)<-1$ and ${\wh f_\ga^\tau}(\kappa)>1$.
\end{remark}

\begin{remark}\label{RemarkBordsInf}
	Taking into account the arguments in Propositions~\ref{proposition2} and \ref{proposition3} as well as the sign of the difference $\widetilde{\gamma}-\gamma$ in \eqref{neg,gpm}, we conclude that the functions ${\wh g_{-}}(\tau;\kappa)$ and ${\wh g_{+}}(\tau;\kappa)$ in \eqref{neg,gpm} are related, respectively, to the newly added eigenvalues induced by the perturbation in the gaps \emph{below} and \emph{above} $\inf\sigma(\H_{\gamma})$.
\end{remark}

\begin{lemma}	\label{monoton}
Let $(\lambda_-,\lambda_+){\subset (0,\infty)\setminus\sigma(\H_\gamma)}$   { (and, hence, $(\lambda_-,\lambda_+)\subset (0,\infty)\setminus\sigma(\H_\gamma^{\theta_2})$ for each $\theta_2\in [-\pi,\pi]$)}. Then for each value of the Floquet parameter $\theta_2\in [-\pi,\pi]$, or equivalently for $\tau:=\cos\theta_2\in [-1,1]$, the functions $k\mapsto g_{\pm}(\tau;k)$ in \eqref{proof,lam=lamp,equiv} are monotonous in $ (\sqrt{\lambda_-},\sqrt{\lambda_+})$. Furthermore, for fixed $k\in (\sqrt{\lambda_-},\sqrt{\lambda_+})$, the functions $[-1,1]\ni\tau\mapsto g_{\pm}(\tau;k)$ are also monotonous.
\end{lemma}
\begin{proof}	
    We use \emph{reduction ad absurdum:} suppose that there is a $\tau\in [-1,1]$ such that $g_{\pm}(\tau;\cdot)$ is not monotonous. Evidently, the function $g(\tau;\cdot)$ is continuous within $(\sqrt{\lambda_-},\sqrt{\lambda_+})$; it may have singularities only at those  $k$ for which $\sin ka$ or $\sin kb$ vanishes, but there are no such points inside the gaps of $\sigma(\H_\ga)$. The continuity of $g_{\pm}(\tau;k)$ and its assumed non-monotonicity imply
	that there exists a
	$\widetilde\gamma\in\mathbb{R}$ and two distinct points $k_1,k_2\in (\sqrt{\lambda_-},\sqrt{\lambda_+})$ such that
	$$
	\widetilde\gamma=g_{\pm}(\tau;k_1)=g_{\pm}(\tau;k_2).
	$$
	Consequently, the numbers $k_1^2$ and $k_2^2$ are two distinct eigenvalues of the operator $\H^{\theta_2}_{\gamma,\widetilde\gamma}$ lying in $(\lambda_-,\lambda_+)$.
	This is not possible, however, because $\H^{\theta_2}_{\gamma,\widetilde\gamma}$ is obtained from $\H^{\theta_2}_{\gamma}$ by a single $\delta$-coupling strength change, and as observed above, such a perturbation may
    produce no more than a \emph{single} eigenvalue in each gap. In this way, we arrive at a a contradiction.

 To prove the monotonicity with respect to $\tau$, we compute
$$\left|{\partial g_{\pm}\over\partial\tau}(\tau,k)\right|=\frac{2k}{|\sin ka|}
 ((f_\ga^\tau(k) )^2-1)^{-1/2}|f_\ga^\tau(k)|\cdot\left| \frac{\sin ka}{\sin kb}\right|.$$
 Since $(\sqrt{\lambda_-},\sqrt{\lambda_+})$ contains no zeros of
 both ${\sin ka}$ and ${\sin kb}$ and $|f_\ga^\tau(k)|>1$ on $(\lambda_-,\lambda_+)$, we get $\left|{\partial g_{\pm}\over\partial\tau}(\cdot,k)\right|>0$, which give the sought claim.
\end{proof}

Similar result holds for negative gaps.
\begin{lemma}	\label{monoton.neg}
Let $(\lambda_-,\lambda_+){ \subset (-\infty,0)\setminus\sigma(\H_\gamma)}$. Then for each  $\tau\in [-1,1]$, the functions $\kappa\mapsto {  \wh g_{\pm}}(\tau;{\black\kappa})$ in {\black\eqref{neg,gpm}} are monotonous in $(\sqrt{|\lambda_+|},\sqrt{|\lambda_-}|)$,
and, for fixed ${\black\kappa}\in (\sqrt{|\lambda_+|},\sqrt{|\lambda_-}|)$, the functions $[-1,1]\ni\tau\mapsto {\black\wh g_{\pm}(\tau;\kappa)}$ are also monotonous.
\end{lemma}

\begin{remark}\label{remark.newbands}
The observations we collected in Remark~\ref{RemarkBords},\,\ref{RemarkBordsInf}
and
Lemmata~\ref{monoton},\,\ref{monoton.neg}
lead us to the following conclusions.
To describe the new band that may appear within a   positive gap $(\lambda_-,\lambda_+)$ in the spectrum of  the unperturbed
operator $\H_\gamma$, it is enough to inspect the behavior of the curves $ {\black g_+(1;k)}$ and $g_+(-1;k)$ (respectively, ${\black g_-(1;k)}$ and $g_-(-1;k)$) provided the condition $f_\gamma^\tau(k)<-1$ (respectively, $f_\gamma^\tau(k)>1$) holds within this gap.
Namely, let $\mathscr{D}$ be a set on the half plane $\{(\wt\ga,k):\ k>0\}$ (see Figures~\ref{threefigsRect}--\ref{threefigsRectGolden3})
squeezed between the horizontal lines $k=\sqrt{\lambda_-}$ and $k=\sqrt{\lambda_+}$ and the curves
$\wt\gamma={\black g_+(1;k)}$ and $\wt\gamma=g_+(-1;k)$ (respectively, $\wt\gamma={\black g_-(1;k)}$ and $\wt\gamma=g_-(-1;k)$); on Figures~\ref{threefigsRect}--\ref{threefigsRectGolden3} the set $\mathscr{D}$ corresponds to one of the blue areas (respectively, the red areas).
The remaining curves $\wt\gamma=g_+(\tau;k)$, $k\in (\sqrt{\lambda_-},\sqrt{\lambda_+})$
(respectively, $\wt\gamma=g_-(\tau;k)$, $k\in (\sqrt{\lambda_-},\sqrt{\lambda_+})$)
with $\tau\in (-1,1)$ fills the interior of $\mathscr{D}$.
Then the intersection of the vertical line $\wt\ga=\mathrm{const}.$ and the set $\mathscr{D}$ is either  empty (cf.~Propositions~\ref{proposition1}--\ref{proposition3}) or
it is compact interval which may
or may not touch the endpoints of  $(\sqrt{\lambda_-},\sqrt{\lambda_+})$; the square of this interval constitutes the new band within the gap $(\lambda_-,\lambda_+)$.
Similarly, to describe the new negative bands, we need to examine the behavior of the curves $\wh g_+(\black 1;\kappa)$ and $\wh g_+(-1;\black\kappa)$ (respectively, $\wh g_-(\black 1;\kappa)$ and $\wh g_-(-1;\black\kappa)$) provided   $\black\wh f_\gamma^\tau(\kappa)<-1$ (respectively, $\black\wh f_\gamma^\tau(\kappa)>1$).
\end{remark}

Now, we can formulate the following proposition concerning those situations in which each spectral gap of $\sigma(\H_\ga)$ is affected by the perturbation.

\begin{proposition}\label{proposition4}
	Apart from the situations mentioned in Propositions~\ref{proposition1}-\ref{proposition3}, there is always a new spectral band resulting from the perturbation in each gap of $\sigma(\H_\ga)$; these new bands may or may not touch the edges of the original bands.
\end{proposition}

 Before to proceed to the proof of the proposition, we
establish the following auxiliary result.

\begin{lemma}\label{lemma:interior}
Let $k_n\ceq \frac{\pi n}{a}$ with $n\in\N$.
Assume that $\sin k_n b\not= 0$, and let
\begin{gather}
\gamma_n^\pm\ceq\pm{2k_n}\left(\tan \frac{k_n b}2\right)^{\pm 1}.
\end{gather}
Then $k_n$ belongs to the interiors of $\sigma(\H_{\ga_n^+})$ and $\sigma(\H_{\ga_n^-})$.
\end{lemma}

\begin{proof}
It is enough to proof the lemma for $\gamma_n^+$; for $\gamma_n^-$ the proof
is similar.
We denote:
\begin{gather}
A(k)\ceq \frac{\gamma_n^+ \, \sin ka}{2 k}+\frac{ \sin k(a+b)}{\sin k b},\quad  B(k)\ceq \frac{\sin ka }{\sin kb},\quad C(k)=1+|B(k)|-|A(k)|.
\end{gather}
To prove the sought claim, we have to show (cf.~\eqref{per,band,con}) that
$C(k)\ge 0$ for $k$ belonging to a neighborhood of $k_n$.
Below we consider four cases.

\noindent {\bf Case I:} $\sin k_n b>0$ and $\cos k_n a =1$
(consequently,  $\sin ka$ changes sign at $k_n$ from minus to plus).
In this case
$\black A(k)\to 1$ as $k\to k_n$, whence, for sufficiently small $\delta_1>0$ we have $|A_n(k)|=A_n(k)$ as $|k-k_n|<\delta_1$.
Furthermore, there is $\delta_2\in (0,\pi)$ such that
$\sin k_n b>0$ for $|k-k_n|<\delta_2$;
consequently
$|B_{n}(k)|=-B_{n}(k)$ for $k\in (k_{n}-\delta_2,k_{n})$ and
$|B_{n}(k)|=B_{n}(k)$ for $k\in (k_{n},k_{n}+\delta_2)$.
As a result, we have
\begin{gather}\label{Ck:1}
C(k)=
\begin{cases}
1-B(k)-A(k),&k\in (k_n-\delta,k_n),\\
1+B(k)-A(k),&k\in (k_{n},k_{n}+\delta),
\end{cases}
\end{gather}
where $\delta=\min\{\delta_1,\delta_2\}$.
By simple manipulations
we get the following asymptotic expansions for even $n$ (i.e., for those $n$ implying $\cos k_n a=1$):
\begin{align} \label{mm-even}
1-B(k)-A(k)&=
-\frac{2a}{\sin k_n b} (k-k_n)+\mathcal{O}((k-k_n)^2),
\\ \label{pm-even}
1+B(k)-A(k)&=
\frac{ \pi n a^2(1 + \cos k_n b) +2a^2(k_n b + \sin k_nb)}{4n\pi \cos^2 \frac{k_nb}2 }(k-k_n)^2+\mathcal{O}((k-k_n)^3)
\end{align}
(not that $\cos  \frac{k_nb}2\not= 0$ if $\sin k_nb\not=0$).
Taking into account that $x>\sin x$ as $x> 0$ and $|\cos x|\le 1$,
we immediately  conclude from \eqref{Ck:1}--\eqref{pm-even} that the function
$C(k)$ is non-negative in a neighborhood of $k_n$.

\noindent {\bf Case II:} $\sin k_n b>0$ and $\cos k_n a =-1$
(consequently,  $\sin ka$ changes sign at $k_n$ from plus to minus).
In this case we have
\begin{gather}\label{Ck:2}
C(k)=
\begin{cases}
1+B(k)+A(k),&k\in (k_n-\delta,k_n),\\
1-B(k)+A(k),&k\in (k_{n},k_{n}+\delta).
\end{cases}
\end{gather}
For odd $n$ (i.e., for those $n$ implying $\cos k_n a=-1$) we have:
\begin{align}
\label{pp-odd}
1+B(k)+A(k)&=-\frac{2a}{\sin k_n b} (k-k_n)+\mathcal{O}((k-k_n)^2),
\\
\label{mp-odd}
1-B(k)+A(k)&=
\frac{ \pi n a^2(1 + \cos k_n b) +2a^2(k_n b + \sin k_nb)}{4n\pi \cos^2 \frac{k_nb}2 }(k-k_n)^2+\mathcal{O}((k-k_n)^3).
\end{align}
The expansions \eqref{pp-odd}--\eqref{mp-odd} together with \eqref{Ck:2} imply $C(k)\ge 0$ for $|k-k_n|$ sufficiently small.

\noindent {\bf Case III:} $\sin k_n b<0$ and $\cos k_n a =1$.
Here we have
$C(k)=1+B(k)-A(k)$ for $k\in (k_n-\delta,k_n)$
and $C(k)=1-B(k)-A(k)$ for $k\in (k_{n},k_{n}+\delta)$ and a sufficiently small $\delta$. Using these facts and the expansions \eqref{pm-even} and \eqref{mm-even},
we conclude that $C(k)\ge 0$ in a neighborhood of $k_n$.

\noindent {\bf Case IV:} $\sin k_n b<0$ and $\cos k_n a =-1$.
One has
$C(k)=1-B(k)+A(k)$ for $k\in (k_n-\delta,k_n)$
and $C(k)=1+B(k)+A(k)$ for $k\in (k_{n},k_{n}+\delta)$ and a sufficiently small $\delta$. Then, by the expansions \eqref{mp-odd} and \eqref{pp-odd},
we derive $C(k)\ge 0$ in a neighborhood of $k_n$.
This ends the proof of the lemma.
\end{proof}

\begin{proof}[Proof of Proposition~\ref{proposition4}]
To prove the claim, we inspect the behavior of the expressions $g_{\pm}(\tau;k)$ and ${ \wh g_{\pm}}(\tau;\kappa)$ with $\tau=\pm1$ which, as mentioned above (cf.~Remark~\ref{remark.newbands}), describe the spectral band edges in $S\gaga$\,, at the values of $k$ and $\kappa$ referring to the edges of the original bands of $\sigma(\H_\ga)$. As before, we consider separately the positive and negative spectrum.

\textit{\textbf{Positive spectrum.}} As noted in Remark~\ref{rem:periodic}, the edges of the spectral bands in $\sigma(\H_\ga)$ are characterized by two types of conditions; denoting them by \ref{caseI} and \ref{caseII}, they are:

	\begin{enumerate}[label=\textnormal{(\Roman*)}]
		\setlength{\itemsep}{-3pt}
		\item \label{caseI} The edges that are explicitly given by $\sin ka=0$ and $\sin kb=0$, their energies being $(\frac{n\pi}{a})^2$ or $(\frac{n\pi}{b})^2$ with some $n\in\mathbb{N}$. They are the upper (respectively, lower) edges of the spectral bands for $\gamma>0$ (respectively, $\gamma<0$), cf. Figs.~\ref{threefigsRect}--\ref{threefigsRectGolden3}; recall that the bands referring to $g_{+}(\tau;k)$ and $g_{-}(\tau;k)$ are, respectively, marked in blue and red. We observe two different behavior types of the functions $g_{\pm}(\pm1;\cdot)$ at the mentioned band edges of $\sigma(\H_\ga)$. First, it may happen that both functions $g_{+}(\pm1;\cdot)$ (respectively, $g_{-}(\pm1;\cdot)$) tend asymptotically to $\pm\infty$ as the energy changes, and alternatively, it may happen that only one of the functions $g_{+}(\pm1;\cdot)$ (respectively, $g_{-}(\pm1;\cdot)$) behaves in this way, while the other reaches the edges of the original bands.
		
		It is easy to see that the \emph{first situation} occurs \emph{only} around the edges determined by the condition $\sin ka=0$ (even if $\sin kb$ vanishes simultaneously). To this aim, let us inspect the behavior of the functions $g_{\pm}(\tau;\cdot)$ with $\tau=-1$ and $\tau=1$ as $k\to \frac{n\pi}{a} $, $n\in\mathbb{N}\,$. For the sake of convenience, it suffices to consider only  the absolute value of the second term in $g_{\pm}(\tau;k)$ in \eqref{proof,lam=lamp,equiv}, that is
		\begin{equation}\label{proof,infinity}
			h_{\pm}(k) :=\frac{2k}{\sin ka}\sqrt{ \left( \frac{\gamma \, \sin ka}{2 k}+\frac{ \sin k(a+b)}{\sin k b}\pm \frac{\sin ka }{\sin kb}\right)^2-1}\, ,
		\end{equation}
		where the lower and upper signs correspond to $\tau=1$ and $-1$, respectively. For further convenience, let us consider the squared form of the functions $h_{\pm}(k) $; by simple manipulations, we get
		\begin{equation}\label{proof,infinity,2}
			h_{\pm}^2(k) = 4k^2\left[ \left(  \frac{\gamma}{2k}\pm (\tan \frac{kb}2)^{\mp1}\right)^2 +2\frac{\cos ka }{\sin ka} \left(  \frac{\gamma}{2k}\pm (\tan \frac{kb}2)^{\mp1}\right)-1 \right].
		\end{equation}
If  $\sin kb$ tends to zero as  $k\to \frac{\pi n}a $ (i.e, $\sin \frac{\pi n b}a=0$),
then the expressions $|\frac{\gamma}{2k}\pm (\tan \frac{kb}2)^{\mp1}|$ tend  either to
$\frac{\gamma}{2k}$ or to $\infty$; hence in this case $h_{\pm}^2(k)\to \infty$ as  $k\to \frac{\pi n}a $.
If $\sin \frac{\pi n}b \not=0$, then $\frac{\gamma}{2k}\pm (\tan \frac{kb}2)^{\mp1}$ do not tend to zero
(otherwise, by Lemma~\ref{lemma:interior}, $\frac{\pi n}a$ is an internal point of $\sigma(\H_\gamma)$);
thus, in this case, $h^2_\pm(k)$ tends to infinity too. This, together with the claims of Proposition~\ref{proposition1}-\ref{proposition3} as well as of Lemma~\ref{monoton} concerning the monotonicity of $g_{\pm}(\tau;\cdot)$ in each gap of $\sigma(\H_\ga)$, implies that approaching the values $k=\frac{n\pi}{a}$, the expressions $g_{\pm}(\tau;k)$ tend asymptotically to plus and minus infinity for $\gamma<0$ and $\gamma>0$, respectively, cf. Figs.~\ref{threefigsRect}--\ref{threefigsRectGolden3}.

The \emph{second situation} happens \emph{only} at those band edges of $\sigma(\H_\ga)$ given by the condition $\sin kb=0$ while $\sin ka\neq0$; needless to say, the case $\sin ka=0$ belongs to the first situation described above. To show this, denote $ k_o:=\frac{(2n-1)\pi}{b}$ and $ k_e:=\frac{2 n\pi}{b}$ with $n\in\mathbb{N}$, where the subscripts $o$ and $e$ stand for $odd$ and $even$, respectively; then we get easily
		\begin{equation}\label{touchsinkb,odd}
			\begin{aligned}
			&	\lim_{k\to k_o }   g_{\pm}(1;k)=\pm \text{sgn}\left( \csc k_o a \right)  \infty\,,\\[6pt]
			&	\lim_{k\to k_o}   g_{\pm}(-1;k)=\gamma\pm\frac{\csc  k_o a }b\sqrt{\rho } =:\widetilde{\gamma}_o\,,
			\end{aligned}
		\end{equation}
and		
		\begin{equation}\label{touchsinkb,even}
			\begin{aligned}
			&	\lim_{k\to k_e}   g_{\pm}(-1;k)=\pm \text{sgn}\left( \csc k_e a \right)  \infty    \,,\\[6pt]
			&	\lim_{k\to  k_e}   g_{\pm}(1;k)=\gamma\pm\frac{\csc k_e a}b\sqrt{\sigma}=:\widetilde{\gamma}_e     \,,
			\end{aligned}
		\end{equation}
where
		\begin{align*}
			&	\rho:=(b \gamma   \sin  k_o a)^2+ ((4 n-2) \pi  \cos k_o a)^2 +(4 n-2) \left(b \gamma \, {\sin k_o a} -(4 n-2)\pi\right)\pi \,, \\
			&   \sigma:=(b \gamma   \sin  k_e a)^2+(4 n\pi   \cos  k_e a)^2+4 n\pi \left(b \gamma  \sin k_e a-4 \pi  n\right).
		\end{align*}
Naturally, the signs $\pm$ in Eqs.~\eqref{touchsinkb,odd} and \eqref{touchsinkb,even} correspond mutually at both sides. Consequently, taking into account the claims of Propositions~\ref{proposition1}-\ref{proposition3} as well as of Lemma~\ref{monoton}, we conclude that one of the functions $g_{+}(\pm1;\cdot)$, and similarly one of the functions $g_{-}(\pm1;\cdot)$, tends asymptotically around the points $ k=\frac{n\pi}{b}$, $n\in\mathbb{N}$, to $+\infty$ (respectively, $-\infty$) if $\gamma<0$ (respectively, $\gamma>0$), while the other reaches the finite value $\widetilde{\gamma}_o$ or $\widetilde{\gamma}_e $. Furthermore, having in mind that $\widetilde{\gamma}=g_{\pm}(\tau;k)$, this implies that at $k=\frac{n\pi}{b}$, if $\gamma>0$ (respectively, $\gamma<0$), the lower (respectively, upper) edges of the new spectral bands of $S\gaga$ always touch the upper (respectively, lower) edges of the original bands provided $\widetilde{\gamma}<\widetilde{\gamma}_o$ or $\widetilde{\gamma}<\widetilde{\gamma}_e$ (respectively, $\widetilde{\gamma}>\widetilde{\gamma}_o$ or $\widetilde{\gamma}>\widetilde{\gamma}_e$), cf. Figs.~\ref{threefigsRect}--\ref{threefigsRectGolden3}.

		\medskip

\item \label{caseII} The band edges that are characterized by the conditions
		\begin{equation}\label{class2}
			\frac{\gamma \, \sin ka}{2 k}+\frac{ \sin k(a+b)}{\sin k b}=\pm 1 \pm \left\lvert \frac{\sin ka }{\sin kb}\right\rvert ,
            \end{equation}
		in which $\sin ka$ and $\sin kb$ are nonzero; more precisely, they correspond to the lower (respectively, upper) edges of the spectral bands of $\sigma(\H_\ga)$ for $\gamma>0$ (respectively, $\gamma<0$). As can be seen in Figs.~\ref{threefigsRect}--\ref{threefigsRectGolden3}, in such a case, one of the functions $g_{+}(\pm1;\cdot)$, and similarly one of the functions $g_{-}(\pm1;\cdot)$, touches the original bands at $\widetilde{\gamma}=\gamma$ while the other touches them at $\widetilde{\gamma}>\gamma$ (respectively, $\widetilde{\gamma}<\gamma$) if $\gamma<0$ (respectively, $\gamma>0$).

		To show that, we mimick the argument used for type~\ref{caseI} and inspect the behavior of the functions $g_{\pm}(\tau;\cdot)$ with $\tau=\pm1$, at the values of $k$ satisfying the original band edge conditions \eqref{class2}.
        To this aim, recall first that by Remark~\ref{RemarkBords}, $g_{+}(\tau;k)$ and $g_{-}(\tau;k)$ are, respectively, associated with the constraints $f_\ga^\tau(k)<-1$ and $f_\ga^\tau(k)>1$. This, together with simple manipulations, brings us to the conclusion that the minus and plus signs in $g_{\pm}(\tau;k)$ in \eqref{proof,lam=lamp,equiv} are fulfilled \emph{only} with the upper and lower signs in \eqref{class2}, respectively, i.e., if the gap is characterized by $\black f_\ga^\tau(k)>1$ (respectively, $f_\ga^\tau(k)<-1$), then the equality \eqref{class2} with plus sign (respectively, minus sign) holds on the band edge). Substituting then from \eqref{class2} into \eqref{proof,lam=lamp,equiv} and paying a proper attention to the signs of the respective conditions, we arrive at
		\begin{equation}\label{simp,lam=lamp,rec,sub}
			\widetilde{\gamma} =\gamma\pm\frac{2k}{\sin ka}\sqrt{ \left( \mp 1 \mp \left\lvert \frac{\sin ka }{\sin kb}\right\rvert-\tau\, \frac{\sin ka }{\sin kb}\right)^2-1}=:g^{ \rm e}_{\pm}(\tau;k)\, ,
		\end{equation}
in which we have used the superscript $\rm e$ for $g_{\pm}(\tau;k)$ to emphasize that they correspond to the functions values at the original band edges characterized by \eqref{class2}. Now, depending on the sign of the function $  \frac{\sin ka }{\sin kb} \,$, we get
	        \begin{equation}\label{secondii2}
			\begin{aligned}
				&	g^{ \rm e}_{\pm}(-1;k)=\gamma\,, \quad\quad
				g^{ \rm e}_{\pm}(1;k)=\gamma \pm \frac{2k}{\sin ka}\sqrt{\left(1\pm 2\frac{\sin ka }{\sin kb}\right)^2-1}=:\widetilde{\gamma}^{p\pm}_1\,, \\
				& g^{ \rm e}_{\mp}(1;k)=\gamma\,, \quad\quad
				g^{ \rm e}_{\mp}(-1;k)=\gamma \mp \frac{2k}{\sin ka}\sqrt{\left(1\pm 2\frac{\sin ka }{\sin kb}\right)^2-1}=:\widetilde{\gamma}^{p\mp}_2\,,		
			\end{aligned}
		\end{equation}
		where the upper and lower signs correspond, respectively, to the plus and minus sign of $ \frac{\sin ka }{\sin kb} $. This confirms the observations that one of the functions $g_{+}(\pm1;\cdot)$, and similarly one of the functions $g_{-}(\pm1;\cdot)$, always touches the edges of the original bands at $\widetilde{\gamma}=\gamma$ while the other touches them at $\widetilde{\gamma}=\widetilde{\gamma}^{p\pm}_1$ or $\widetilde{\gamma}=\widetilde{\gamma}^{p\mp}_2$. Hence, in accordance with the claims of Propositions~\ref{proposition1}-\ref{proposition3} and Lemma~\ref{monoton}, we conclude that for $\gamma>0$ (respectively, $\gamma<0$), the upper (respectively, lower) edges of the new spectral bands do not touch the lower (respectively, upper) edges of the original bands if $\widetilde{\gamma}< \widetilde{\gamma}^{p\pm}_1 $ or  $\widetilde{\gamma}< \widetilde{\gamma}^{p\mp}_2 $ (respectively, $\widetilde{\gamma}>\widetilde{\gamma}^{p\pm}_1$ or $\widetilde{\gamma}>\widetilde{\gamma}^{p\mp}_2 $), while on the other hand, if $\widetilde{\gamma}^{p\pm}_1\leq\widetilde{\gamma}\leq\gamma $ or $\widetilde{\gamma}^{p\mp}_2\leq\widetilde{\gamma}\leq\gamma $ (respectively, $\gamma\leq\widetilde{\gamma}\leq\widetilde{\gamma}^{p\pm}_1 $ or $\gamma\leq\widetilde{\gamma}\leq\widetilde{\gamma}^{p\mp}_2 $) they do touch, cf. Figs.~\ref{threefigsRect}--\ref{threefigsRectGolden3}.
	\end{enumerate}

 Now, let us pass to the negative part of the spectrum.	
	
	\textit{\textbf{Negative spectrum.}} It can be dealt with in a similar way, by inspecting the behavior of the functions $\kappa\mapsto {\black \wh g}_{\pm}(\tau;\kappa)$ with $\tau=\pm1$ at the values of $\kappa$ referring to the edges of the single negative band in $\sigma(\H_\ga)$. The latter, from \eqref{per,neg,band,con}, or equivalently from \eqref{class2} with $k$ replaced by $i\kappa$, is given by the conditions
	\begin{equation}\label{IInegBand}
		\frac{\gamma \, \sinh \kappa a}{2 \kappa}+\frac{ \sinh \kappa(a+b)}{\sinh \kappa b}=\pm 1 \pm  \frac{\sinh \kappa a }{\sinh \kappa b} ,
	\end{equation}
	in which the upper and lower signs correspond, respectively, to the lower and upper edges of the original negative band. Mimicking the argument used above for the type~\ref{caseII} of the positive spectrum, or simply by replacing $k$ by $i\kappa$ in \eqref{secondii2} with the upper signs only since $\frac{\sinh \kappa a }{\sinh \kappa b}>0$, we get
	\begin{equation}\label{secondii,neg}
		\begin{aligned}
			&{\black \wh g}^ {\;\rm e}_{+}(-1;\kappa )=\gamma\,, \qquad\quad
			{\black \wh g}^{ \;\rm e}_{+}(1;\kappa )=\gamma + \frac{2\kappa }{\sinh \kappa a}\sqrt{\left(1+ 2\frac{\sinh \kappa a }{\sinh \kappa b}\right)^2-1}=:\widetilde{\gamma}^n_1\,, \\
			& {\black \wh g}^{ \;\rm e}_{-}(1;\kappa )=\gamma\,, \qquad\quad
			{\black \wh g}^{ \;\rm e}_{-}(-1;\kappa )=\gamma - \frac{2\kappa }{\sinh \kappa a}\sqrt{\left(1+ 2\frac{\sinh \kappa a }{\sinh \kappa b}\right)^2-1}=:\widetilde{\gamma}^n_2\,,		
		\end{aligned}
	\end{equation}
where the superscript $ \rm e$ is used again to emphasize that we deal with the functions values at the original band edges characterized by \eqref{IInegBand}. Note that from \eqref{secondii,neg} we easily infer that $\widetilde{\gamma}^n_1>\gamma$ and $\widetilde{\gamma}^n_2<\gamma$ because $\kappa$ and $\sinh \kappa a$ are positive.
	
	On the other hand, one easily checks that in the limit $\kappa\to\infty$, the function $\kappa\mapsto {\black \wh g}_{-}(\tau;\kappa)$ in \eqref{neg,gpm} tends to $-\infty$; recall that the functions ${\black \wh g}_{-}(\tau;\kappa)$ and ${\black \wh g}_{+}(\tau;\kappa)$ correspond, respectively, to the negative eigenvalues of $S\gaga$ in the gaps below and above $\inf\sigma(\H_{\gamma})$, see Remark~\ref{RemarkBordsInf}.
	
	Therefore, in accordance with the Proposition~\ref{proposition1}-\ref{proposition3} as well as Lemma~\ref{monoton}, we conclude that as $\widetilde{\gamma}$ increases from $-\infty$ to $\infty$, the function $\kappa\mapsto {\black \wh g}_{-}(-1;\kappa)$ (respectively, $\kappa\mapsto {\black \wh g}_{-}(1;\kappa)$) grows with respect to $-\kappa^2$ from $-\infty$ to the lower edge of the original negative band and touches it at $\widetilde{\gamma}=\widetilde{\gamma}^n_2$ (respectively, $\widetilde{\gamma}=\gamma$). Similarly, the function $\kappa\mapsto {\black \wh g}_{+}(-1;\kappa)$ (respectively, $\kappa\mapsto {\black \wh g}_{+}(1;\kappa)$) touches the upper edge of the original negative band at $\widetilde{\gamma}=\gamma$ (respectively, $\widetilde{\gamma}=\widetilde{\gamma}^n_1$), and then increases with respect to $-\kappa^2$. As a result, the perturbation-induced negative bands do not touch the lower (respectively, upper) edge of the orginal band if $\widetilde{\gamma}< \widetilde{\gamma}^n_2 $ (respectively, $\widetilde{\gamma}> \widetilde{\gamma}^n_1$). This behavior is illustrated in Figs.~\ref{threefigsRect}-- \ref{threefigsRectGolden3}.

Summing up the discussion of the positive and negative spectrum and using Lemma~\ref{monoton} about the monotonicity of the functions $k\mapsto g_{\pm}(\tau;k)$ and $\kappa\mapsto {\black \wh g}_{\pm}(\tau;\kappa)$ in each gap of $\sigma(\H_\ga)$, we have obtained the sought claim.
\end{proof}

\begin{remark}\label{remarkGapClose}
	Using the results about the positive spectrum contained in Proposition~\ref{proposition4} and Lemma~\ref{monoton}, we infer that some of the original gaps of $\sigma(\H_\ga)$  may get completely closed by the perturbation, cf. Figs.~\ref{threefigsRect}--\ref{threefigsRectGolden3}. Such a situation occurs if $\widetilde{\gamma}_o$ or $\widetilde{\gamma}_e $ is greater (respectively, less) than $\widetilde{\gamma}^{p\pm}_1$ or $\widetilde{\gamma}^{p\mp}_2$ for $\gamma>0$ (respectively, $\gamma<0$). As we mention in the next remark, however, there is an exception to these conclusions.
\end{remark}

\begin{remark}\label{remarkGapCloseSq}
	The exception concerns the particular case of a \emph{square lattice}, $a=b$. The reason is that in contrast to a non-square lattice, the second situation of type \ref{caseI} in Proposition~\ref{proposition4} which gave rise to a gap closure for a general rectangular lattices, leading to the finite values $\widetilde{\gamma}_o$ and $\widetilde{\gamma}_e$ for $g_{\pm}(\tau;k)$ in the limit in Eqs.~\ref{touchsinkb,odd} and \ref{touchsinkb,even} mentioned above,
does not exist anymore. Consequently, \emph{none} of the original gaps in $\sigma(\H_\ga)$ can be closed completely by the perturbation we consider, cf. Fig. \ref{threefigssquare}.	
\end{remark}

To conclude this, admittedly a bit technical discussion, let us summarize the main results of this section:

\begin{theorem}\label{thmGen}
	The spectrum $\sigma(\H_\ga)$ of the periodic rectangular lattice changes under the considered perturbation as follows:
	\begin{enumerate}[label=\textnormal{(\roman*)}]
		\setlength{\itemsep}{-3pt}
		\item \label{th1} If $\gamma<0$ and $\widetilde{\gamma}<\gamma$, there is no new spectral band in original gaps lying above $\inf\,\sigma(\H_{\gamma})$, but there is always a new negative band in the semi-infinite gap below $\inf\,\sigma(\H_{\gamma})$ which may or may not touch the lower edge of the original negative band.
		\item \label{th2} If $\gamma<0$ and $\widetilde{\gamma}>\gamma$,  there is no new negative band below $\inf\,\sigma(\H_{\gamma})$, but in each original gap lying above $\inf\,\sigma(\H_{\gamma})$ there is a new spectral band which may or may not touch the edges of the original bands. In the non-square lattice case, $a\neq b$, it may happen that some of the original gaps get completely closed under the perturbation.
		\item \label{th3} If $\gamma>0$ and $\widetilde{\gamma}<\gamma$, in each original gap of $\sigma(\H_\ga)$ there is a new spectral band which may or may not touch the edges of the original bands. Again, in the non-square lattice case it may happen that some of the original gaps get completely closed under the perturbation.
		\item \label{th4} If $\gamma>0$ and $\widetilde{\gamma}>\gamma$, there is no new spectral band in original gaps of $\sigma(\H_\ga)$; the spectrum remains unchanged under the perturbation.
	\end{enumerate}
\end{theorem}


\section{High-energy regime}
\label{high,en,section}
Let us now finally take a look at the asymptotic behavior of the spectrum in the high-energy regime, $k\to\infty$. For the unperturbed operator $\H_\ga$, the band condition \eqref{per,band,con} reduces to
\begin{equation}\label{per,band,pos,high}
	\left\lvert   \frac{ \sin k(a+b)}{\sin k b} \right\rvert \leq  1+\left\lvert \frac{\sin ka }{\sin kb}\right\rvert ,
\end{equation}
up to a relative error of order $\mathcal{O}(k^{-1})$. The latter is responsible for the existence of gaps; it is easy to check that \eqref{per,band,pos,high} is always satisfied. This shows that the spectrum of $\H_\ga$ is dominated by bands. More specifically, referring to \cite[claim (a6) in Sec.III.A]{EG96}, we know that in the generic case (cf.~Remark~\ref{rem:Bethe-Sommerfeld}) their number is infinite and their widths are asymptotically bounded by $2|\gamma|(a+b)^{-1}+ \mathcal{O}(n^{-1})$, where $n\to\infty$ is the gap index. This also means that the probability that a randomly chosen energy belongs to the spectral bands, introduced by Band and Berkolaiko \cite{BB13} as
\begin{equation}\label{probsigma}
	P_{\sigma}(\H_\ga):=\lim_{K\to\infty} \frac{1}{K}\left|\sigma(\H_\ga)\cap[0,K]\right|,
\end{equation}
is equal to one.

What is more interesting is the `size' of the set $S\gaga$ representing the states transported along the line supporting the perturbation, even for energies contained in the unperturbed spectral bands. We are going to show that this set is significant, even if the set-theoretical addition to the spectrum due to the perturbation discussed in the previous section is globally negligible, $P_{\sigma}(\H\gaga)=P_{\sigma}(\H_\ga)=1$.

\begin{proposition}\label{prob1half}
The probability that a randomly chosen number $k^2 $ belongs to the set $S\gaga$, given by \eqref{probsigma} with $\sigma(\H_\ga)$ replaced by $S\gaga$, is equal to $\frac12$.
\end{proposition}
\begin{proof}
From Proposition~\ref{cor:explicitBands}, or more precisely, from the band conditions \eqref{Pert,BandCon1} and \eqref{Pert,BandCon2}, we find that as $k\to\infty$, a number $k^2$ belongs to the set $S\gaga$ if $k$ satisfies either the condition
\begin{equation}\label{Pert,BandCon1he}
	k\in \left\{ \	\left\lvert   \frac{ \sin k(a+b)}{\sin k b}+1 +\mathcal{O}(k^{-1})\right\rvert   \leq   \left\lvert \frac{\sin ka }{\sin kb}  \right\rvert  \quad \cap \quad  (\widetilde{\gamma}-\gamma )\,\frac{ \sin ka}{2 k}>0  \right\},
\end{equation}
or
\begin{equation}\label{Pert,BandCon2he}
	k\in \left\{ \	\left\lvert   \frac{ \sin k(a+b)}{\sin k b}-1+\mathcal{O}(k^{-1})\right\rvert   \leq   \left\lvert \frac{\sin ka }{\sin kb}  \right\rvert  \quad \cap \quad  (\widetilde{\gamma}-\gamma )\,\frac{ \sin ka}{2 k}<0  \right\}.
\end{equation}
First of all, we note that even though the function $ (\widetilde{\gamma}-\gamma )\,\frac{ \sin ka}{2 k}$ in Eqs.~\eqref{Pert,BandCon1he} and \eqref{Pert,BandCon2he} is not periodic with respect to $k$, the measure of $k$ for which the function is positive or negative is still the same as $\sin ka$ since $k>0$ and $\widetilde{\gamma}-\gamma$ is a fixed real number; in other words, the probability that for a randomly chosen value of $k>0$, the function is positive or negative, is equal to $\frac12$.

It remains to find the probabilities that the first conditions in \eqref{Pert,BandCon1he} and \eqref{Pert,BandCon2he} hold; by simple manipulations, those can be rewritten in the more simplified forms
\begin{equation}\label{hecon1}
\xi_1(k):=\left\lvert	\sin k\left(\frac{a}{2}+b\right) \right\rvert- \left\lvert\sin \frac{ka}{2}\right\rvert\leq0,
\end{equation}
and
\begin{equation}\label{hecon2}
\xi_2(k):=\left\lvert	\cos k\left(\frac{a}{2}+b\right)\right\rvert - \left\lvert\cos  \frac{ka}{2}\right\rvert\leq0,
\end{equation}
respectively, both with the relative error $\mathcal{O}(k^{-1})$ which plays no role because the value of \eqref{probsigma} is determined by the asymptotic behavior only. Let us denote the probability that a randomly chosen $k$ satisfies \eqref{hecon1} and \eqref{hecon2} by $p_1$ and $p_2$, respectively. Note that the validity of the two inequalities is mutually exclusive up a set of measure zero due to the fact that if $|\sin x|\leq|\sin y|$ holds for two real $x$ and $y$, then we have $|\cos x|\geq|\cos y|$ in view of the basic identity $\sin^2 x+\cos^2 x=1$, and as a consequence, $p_1+p_2=1$.

It is thus sufficient to check that the probability that a random $k$ belongs to the sets \eqref{Pert,BandCon1he} and \eqref{Pert,BandCon2he} is $\frac12 p_1 $ and $\frac12 p_2\,$, respectively. This is obvious when $a$ and $b$ are incommensurate, because then one can regard the quantities $k\left(\frac{a}{2}+b\right)=:x$ and $ \frac{ka}2=:y$ in the arguments of the trigonometric functions in \eqref{hecon1} and \eqref{hecon2} with $k$ running over $\R_+$ as a pair of independent identically distributed random variables on the torus $[0,2\pi)^2$, and $\sin ka$ takes a fixed sign with probability $\frac12$. Moreover, the conditions \eqref{hecon1} and \eqref{hecon2} simplify in this case to $|\sin x|<|\sin y|$ and $|\cos x|<|\cos y|$, respectively, which obviously implies $p_1=p_2=\frac12$.

For commensurate $a$ and $b$, on the other hand, the last claim is not true in general; if $a=2b$, for instance, we have $p_1=\frac13$ and $p_2=\frac23$. The probabilities of belonging to the sets \eqref{Pert,BandCon1he} and \eqref{Pert,BandCon2he} are nevertheless still $\frac12 p_j,\, j=1,2$. To see that, note that the functions $\xi_i(k)$, $i=1,2$ are periodic with the same period $T$, and as a direct consequence of their evenness, they have mirror symmetry with respect to the midpoint of each cycle. Furthermore, one can easily show that $T$ is always an integer multiple of $t=\frac{2\pi}a$. 
To be more explicit, the first term in $\xi_i(k)$ has the period $T_1=\frac{2\pi }{a+2b} = Rt $ where $R=\frac{a}{a+2b}$ is by assumption a rational number which we can write as $R=\frac mn$, where $m$ and $n$ are integers; the second term in $\xi_i(k)$ has the period $T_2=t$. Therefore, the common period $T$ is $T=qt$ where $q$ is the least common integer multiple of $\frac mn$ and $1$, naturally an integer number. In combination with the fact that the function $k\mapsto \sin ka$, determining the sign of function $(\widetilde{\gamma}-\gamma )\,\frac{ \sin ka}{2 k}$ in \eqref{Pert,BandCon1he} and \eqref{Pert,BandCon2he}, is repeated over intervals of length $T$ and the above mentioned mirror symmetry of $\xi_i(k)$ with respect to the points $k=\frac T2$ (mod $T$) we arrive at the `one-half' claim for commensurate edges.
\end{proof}

\subsection*{Data availability statement}

Data are available in the article.

\subsection*{Conflict of interest}

The authors have no conflict of interest.


\subsection*{Acknowledgments}
M.B. and A.Kh. were supported by the Czech Science Foundation within the project 22-18739S. The work of P.E. was upported by by the European Union's Horizon 2020 research and innovation programme under the Marie Sklodowska-Curie grant agreement No 873071.

\end{document}